\documentclass[11pt, leqno]{article}

\usepackage{amsmath,amsfonts,amssymb,amsthm,amscd}
\usepackage{subfigure}
\usepackage{graphicx}

\newtheorem{cor}{Corrolary}[section]
\newtheorem{lemma}{Lemma}[section]
\newtheorem{theorem}{Theorem}[section]

\theoremstyle{definition}

\theoremstyle{Proposition}
\newtheorem{pro}{Proposition}[section]

\numberwithin{equation}{section}

\date{}
\begin{document}

\title{\bf Asymptotic behaviour of the doubly nonlinear equation $u_t=\Delta_p u^m$ on bounded domains }
\author{Diana Stan$^{\,a}$
~and~ Juan Luis Vazquez$^{\,b}$}

\maketitle

\begin{abstract}
We study the homogeneous Dirichlet problem for the doubly nonlinear equation $u_t = \Delta_p u^m$, where $p>1,\ m>0$ posed in a bounded domain in $\mathbb{R}^N$ with homogeneous boundary conditions and with non-negative and integrable data. In this paper we consider  the degenerate case $m(p-1)>1$ and the quasilinear case $m(p-1)=1$. We establish the large-time behaviour by proving the uniform convergence to a unique asymptotic profile and we also
give rates for this convergence.
\end{abstract}

\tableofcontents

\section{Introduction}
\quad We are interested in describing the behaviour of non-negative solutions of the homogenous Dirichlet problem for the doubly nonlinear equation (DNLE) for
large times. To be precise, we consider the following initial and boundary value problem
\begin{equation}\label{DNLE}
  \left\{ \begin{array}{ll}
  u_{t}(t,x) = \Delta_p u^m(t,x) &\text{for }t>0 \text{ and } x \in \Omega, \\
  u(0,x)  =u_0(x) &\text{for } x \in \Omega, \\
  u(t,x)=0   &\text{for }t>0 \text{ and } x \in \partial \Omega.
    \end{array}
    \right.
\end{equation}
for $m>0$, $p>1$. The problem is posed in a bounded domain $\Omega \in \mathbb{R}^N$ with initial data $u_0 \geq 0$, $u_0 \in L^1(\Omega)$ so that the solution $u(x,t)\ge 0$ too. The $p$-Laplacian operator is well-known to be defined as  $\Delta_p w := \text{div}(|\nabla w|^{p-2}\nabla w)$. We study the large time asymptotic behaviour of solutions to Problem \ref{DNLE} in the ``degenerate  case'' $m(p-1)>1$, also known as slow diffusion case, and  in the ``quasilinear case'' $m(p-1)=1$.

Let us first make some comments  concerning the range $m(p-1)>1$. When $p=2$ we recover the porous medium equation (PME) $u_t=\Delta  u^m$ with $m>1$ while, when $m=1$, we recover the degenerate $p$-Laplacian equation (PLE) $u_t=\Delta_p u$ with $p>2$, both well known equations in the literature. Notice that in this paper we only require $m(p-1)>1$, that also includes  cases where either $m\le 1$ or $ p\le 2$.
The PLE and the PME, as prototypes for degenerate diffusion, enjoy many common properties, such as finite speed of propagation and the existence of some special (self-similar) solutions, which play an important role in describing the asymptotic behaviour for general initial data. In this paper we complete the panorama by analyzing in detail the large-time properties of the degenerate DNLE, which combines the difficulties of both equations and offers some new challenges.

The quasilinear case $m(p-1)=1$ is also interesting to study since it inherits some common features of the Heat Equation, $u_t=\Delta u$ (which can be recovered when $m=1$ and $p=2$): this equation is invariant under scalar multiplication, and it is known that a general solution converges after rescaling to one of the (stationary) solutions of the eigenvalue problem for the $p-$Laplacian operator. However, when $(m,p)\neq (1,2)$ differences appear at the level of regularity and qualitative behaviour. While solutions of the HE are $C^{\infty}$ smooth, solutions of the DNLE have limited regularity due to the degenerate (singular) parabolic character of the equations at the level $u=0$ (see Fig. \ref{figparameters}).

 The remaining ``fast diffusion case'' $m(p-1)<1$ has quite different properties and deserves a separate study. Indeed,  we deal in this case with singular diffusions, and new phenomena  appear such as extinction in finite time, or lack of uniqueness of the asymptotic profile. All this  gives a different flavor to the analysis of the asymptotic behaviour.

As references for the previous theory for  the DNLE we mention \cite{ManfrediVespri} for the degenerate and quasilinear cases and \cite{MR1285092} for the singular case. We mention also that the asymptotic behaviour of the Cauchy problem on $\mathbb{R}^N$ has been studied in \cite{MR2602927}. Many of our results are new even in the $p$-Laplacian case $m=1$, $p>2$.
We also remark that most of the techniques needed to prove existence, uniqueness and other basic properties of the parabolic DNLE flow, can be taken from the books \cite{JLVSmoothing, JLVPorous} for the PME, and  \cite{DiBenedetto} for the PLE.
We also refer to \cite{Aronson1981378} and \cite{JLVsurvey} for a complete asymptotic analysis of the Dirichlet problem on bounded domains, for the PME when $m>1$.

\begin{figure}\label{figparameters}
  \centering
  \includegraphics[width=85mm,height=55mm]{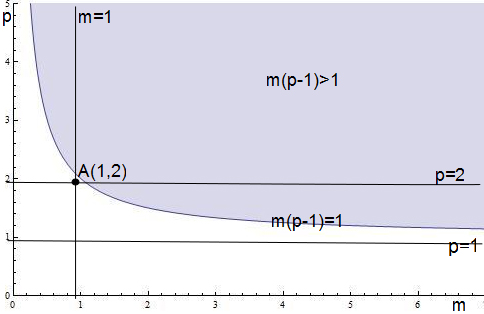}
  \caption{Ranges of parameters $m$ and $p$}
\end{figure}

\medskip

\noindent \textbf{Presentation of the main results.}
The purpose of this work is to analyze completely the asymptotic behaviour of the DNLE on Euclidean bounded domains. For convenience we assume that the boundary $\partial \Omega$ is $C^{2,\alpha}$ smooth. Since the cases $m(p-1)>1$ and $m(p-1)=1$ involve different techniques, we will present them separately.

\medskip

\noindent \textbf{Ia. The degenerate case $m(p-1)>1$.}
This work generalizes the asymptotic analysis carried out in the above mentioned papers \cite{Aronson1981378,JLVsurvey}. The outline of the theory is similar but the double nonlinearity asks for a number of interesting techniques. Throughout the study we will fix the notation $\mu=1/(m(p-1)-1)>0$, since this quantity will appear frequently.

 The asymptotic behaviour is better understood via the well-known method of rescaling and time transformation; let us introduce
\begin{equation}\label{V}
v(\tau,x)= t^{\mu}u (t,x), \quad t = e^{\tau}.
\end{equation}
In this way, Problem \ref{DNLE} is transformed into
\begin{equation}
  \left\{ \begin{array}{ll}
\displaystyle{ v_{\tau}(\tau,x)=\Delta_p v^m(\tau,x) + \mu\,v(\tau,x)},  &\text{for }\tau\in \mathbb{R} \text{ and } x \in  \Omega,  \\[2mm]
v(\tau,x)=0, &\text{for }\tau\in \mathbb{R} \text{ and } x \in \partial \Omega, \\
v(0,x)=v_0 &\text{for }x\in \Omega.
    \end{array}\right.
\end{equation}
In Section \ref{SectionAsympBehDNLE} we prove Theorem \ref{ThAsympBehDNLE}, which shows uniform convergence of the rescaled solution $v(\tau,x)$ to its unique asymptotic profile $f(x)$, as $\tau \rightarrow +\infty$. The stationary profile $f$ can be characterized as the positive solution to the corresponding stationary problem
\begin{equation*}\label{eq_f_degenerate_case}
\Delta_p f^m + \mu\,f=0, \text{ in }\Omega, \quad f=0 \text{ on }\partial \Omega.
\end{equation*}
The result of this Theorem is not surprising, but it does not appear explicitly in literature and it is needed to prove the next results. The techniques used in this step follow the work \cite{JLVsurvey} for the PME.

In Section \ref{Section_Rate_DNLE} we prove sharp rates of convergence of $v(\cdot,\tau)\to f$ as $\tau\to \infty$; this represents the first important result of this paper.

\medskip

\begin{theorem}\label{Th_rate_conv_DNLE}\textbf{(Weighted rate of convergence)}
Let $\Omega\subseteq {\mathbb R}^N$ be a  bounded domain of class $C^{2,\alpha}$, $\alpha>0$. Let $u(t,\cdot)$ be the weak solution to Problem {\rm \ref{DNLE}} corresponding to a nonnegative initial datum $u_0\in L^1(\Omega).$ Then for every $t_0>$ fixed there exists $C>0$ such that the following inequality holds
\begin{equation}\label{ineq_rate_conv}
\left|(1+t)^{\mu} u(t,x)-f(x)\right| \leq Cf(x)(1+t)^{-1} \quad \text{for all }t\ge t_0 \mbox{ and } x\in\Omega,
\end{equation}
where $C$ depends only on $p,m,N,u_0,\Omega$ and $t_0$.
\end{theorem}

In other words
\begin{equation}
u(t,x)=t^{-\mu}f(x)+O(f(x)\,t^{-\mu-1}).
\end{equation}

\noindent\textbf{Remarks. }(i) \textit{Sharpness. }Estimate \eqref{ineq_rate_conv} is sharp, since it is satisfied with equality when we consider the separate variable solution,
\begin{equation}\label{svs}
U(t,x;s)=(s+t)^{\mu}f(x) ,
\end{equation}
with parameter $s>0$.

\medskip
\noindent (ii)\textit{Convergence in relative error. }Let $U(t,x):=U(t,x;0)$  be the separate variable solution \eqref{svs} and let
\[
v(\tau,x)= t^{\frac{1}{m(p-1)-1}}u (t,x)\,, \qquad t = e^{\tau}
\]
be the \emph{rescaled solution} given in \eqref{rescaledEq_DNLE}. We can rewrite \eqref{eq_rate_conv_DNLE} in the following form:
\medskip

\begin{cor}\textbf{(Convergence in relative error)}
Under the assumptions of Theorem {\rm \ref{Th_rate_conv_DNLE}}, if $u$ denotes the solution of Problem {\rm \ref{DNLE}}, then
\begin{equation}
\left\| \frac{u(t,\cdot)}{U(t,\cdot)} -1 \right\|_{L^{\infty}(\Omega)}=
\left\| \frac{v(\tau,\cdot)}{f(\cdot)} -1 \right\|_{L^{\infty}(\Omega)}\le \mathcal{C}t^{-1}.
\end{equation}
\end{cor}

The main idea will be to compare the general solution $u$ of Problem \ref{DNLE} with functions belonging to special families, more exactly self-similar solutions of the DNLE.
We will try to follow the strategy used in the papers \cite{Aronson1981378} and \cite{JLVsurvey} for the case $m>1$ of the PME, and solve the problems caused by the nonlinearity of the $p$-Laplacian operator.

\medskip

\noindent {\bf Ib. Selfsimilar study.} In the process of proving the above results we became interested in the existence and properties of self-similar solutions of the DNLE, that is, functions of the form
$$
\mathcal{U}(t,x)=(t+s)^{-\alpha}h(r),\quad r=|x|(t+s)^{-\beta},
$$
where $\alpha$, $\beta$ are positive parameters and $h$ a real valued function satisfying a certain ODE. As a by-product we give a formal characterization of such solutions. Selfsimilar solutions are often used as a way of indicating the behavior of a general solution of the DNLE.

\medskip

\noindent \textbf{II. The quasilinear case.} In Section \ref{Section_Quasilinear} we study the asymptotic behaviour of solutions of the DNLE when $m(p-1)=1$. Our study uses the preliminary work  \cite{ManfrediVespri} and requires a delicate barrier technique inspired from the work \cite{BDGV09} on fast diffusion stabilization. To be precise, we consider the rescaling
\begin{equation}
v(t,x)=e^{\lambda_1 t}S(x),
\end{equation}
where $\lambda_1$ is the first eigenvalue of the $p$-Laplacian operator $\Delta_p$, \cite{Anane, {Lindqvist}}. Notice that here there is no time transformation, but only rescaling.

This problem was previously studied by Manfredi and Vespri in \cite{ManfrediVespri},  where the authors obtained the convergence, along time subsequences, of $v(t,x)$ to a possible asymptotic profile. At the same time they proved that the set of asymptotic profiles is included in the set of solutions of the corresponding elliptic problem
\begin{equation}\label{eigenvaluePlaplace}
-\Delta_p V = \lambda_1 V^{p-1} \text{ in } \Omega,  \quad V=0 \text{ on } \partial \Omega,
\end{equation}
and moreover $V>0$ in $\Omega$. It is known that when $\lambda=\lambda_1$ then the set of solutions is a linear set, i.e., they have the form $\{ cV_1: c>0\}$, where $V_1$ is a particular normalized solution (a normalized $p$-eigenfunction), cf.  \cite{Anane,Lindqvist}.

In this work, we complete the asymptotic analysis by proving uniform convergence of the rescaled solution  $v(x,t)$ to an unique asymptotic profile; this happens for all times $t\to\infty$ and we also prove a relative error version for this convergence. Our main result in the quasilinear case is the following.

\medskip

\begin{theorem}\label{Th_conv_Rel_error_quasilinear}
Consider $m(p-1)=1$. Let $\Omega\subseteq {\mathbb R}^N$ be a bounded connected domain of class $C^{2,\alpha}$, $\alpha>0$. Let $u(t,\cdot)$ be a weak solution to the Dirichlet Problem {\rm \ref{DNLE}} corresponding to the nonnegative initial datum $u_0\in L^1(\Omega)$.
Let $v(t,x)=e^{\lambda_1 t}u(x,t)$. Then there exists a unique constant $c>0$ such that
\begin{equation}\label{convREFquasilineal}
\lim_{t \rightarrow \infty}\left\| \frac{u(t,\cdot)}{\mathcal{U}(t,\cdot)}-1 \right\|_{L^{\infty}(\Omega)}=\lim_{t\rightarrow \infty}\left\| \frac{v(t,\cdot)}{S(\cdot)} -1\right\|_{L^{\infty}(\Omega)}=0,
\end{equation}
where ${\cal U}(x,t)=e^{-\lambda_1 t}S(x)$  and   $S^m=c_*f$.
\end{theorem}

In order to clarify the result, let us point out that the result states that there is a unique asymptotic profile of the form $S=V^{1/m}$, where $V$ is one of the positive solutions of \eqref{eigenvaluePlaplace}. In other words, there is a unique $c=c(u_0)>0$ such that $V=c V_1 =:c f^m$. Though the asymptotic constant $c$ depends on the data, there is no explicit or semi-explicit formula to compute it. This is a typical occurrence issue of nonlinear evolution problems, see a discussion of the issue in \cite{Kamin_Peletier_Vazquez_1991} when studying the Barenblatt  equation for elastoplastic filtration, a quite different model of nonlinear heat flow.
In order to prove Theorem \ref{Th_conv_Rel_error_quasilinear}, the methods used in the degenerate case do not work anymore and therefore we apply a different method, a barrier argument, based on the one used in \cite{BDGV09} to prove convergence in relative error for the fast diffusion equation.

\section{Asymptotic behaviour for $m(p-1)>1$}\label{SectionAsympBehDNLE}

\subsection{Preliminaries}
In order to present the asymptotic behaviour, we first need to introduce some preliminary results concerning
 the smoothing effects of the DNLE.
\medskip

\noindent {\sc Notations.}  $$Q_T=(0,T)\times \Omega, \ Q=(0,\infty)\times \Omega, \ d(x)=\text{dist}(x,\partial \Omega).$$

The notion of weak solution is defined in the standard sense, we refer to \cite{DiBenedetto}. In addition, it is known by standard semigroup theory that there exists a unique non-negative weak solution $u$ of Problem \ref{DNLE} with good regularity properties and satisfies the Maximum Principle.

Now, we illustrate specific properties concerning the smoothing effects of the DNLE, properties that will be needed
in our proofs (we refer for example to \cite{JLVSmoothing}). In what follows we assume that $\Omega\subseteq {\mathbb R}^N$ is a  $C^{2,\alpha}$ domain, for some $\alpha\in(0,1)$. Let $u(t,\cdot)$ be a weak solution to Problem \ref{DNLE} corresponding to the nonnegative initial datum $u_0\in L^1(\Omega).$

\noindent  \textbf{1. B\'{e}nilan-Crandall type estimates}
\begin{enumerate}
  \item If $m(p-1)>1$, then
\begin{equation}\label{estim_derivative1_DNLE}
u_{t} \geq - \mu t^{-1} u
\end{equation}
in the sense of distributions.
  \item If $m(p-1)<1$, then
\begin{equation}\label{estim_derivative2_DNLE}
u_{t} \leq  \mu t^{-1} u
\end{equation}
in the sense of distributions.
\end{enumerate}

  Also, in the case $m(p-1)>1$, the weak solution $u$ of Problem \ref{DNLE} verifies the following estimate
\begin{equation}\label{Benilan_Crandall_2_DNLE}
\|u_{t}(t+s,x) \|_{1} \leq  \mu (t+s)^{-1} \|u(s)\|_1.
\end{equation}
This inequality is a property of viewing the solution $u(t)$ of Problem \ref{DNLE} with initial data $u_0$ as a semigroup $u(t)=T_t(u_0)$. As an immediate consequence we observe that
\begin{equation}\label{Benilan-Crandall_DNLE}
\|u_{t}(t,x) \|_{1} \leq \mu t^{-1}\|u_0\|_1
\end{equation}
In addition, using estimate \eqref{Benilan_Crandall_2_DNLE} with $t/2$ instead of $t$ and $s$ one can obtain that
\begin{equation}\label{est_derivative_u}
\|u_{t}(t,x) \|_{1} \leq  \mu t^{-1} \|u(t/2)\|_1.
\end{equation}

\noindent  \textbf{2. Smoothing effects}

\noindent In the case $m(p-1)>1$, the solution $u$ of Problem \ref{DNLE} satisfies the following
\begin{equation}\label{est_norm_Lr_DNLE}
\|u(t,\cdot)\|_{L^{r}(\Omega)} \leq C t^{-\mu}, \ t \in (0,+\infty), \ r\geq 1.
\end{equation}
Combining the previous estimates we obtain the absolute bound
\begin{equation}\label{smoothing3_DNLE}
\|u(t,\cdot)\|_{L^{\infty}(\Omega)} \leq C t^{-\mu}, t \in (0,+\infty).
\end{equation}

As a consequence, we can improve the estimate \eqref{Benilan_Crandall_2_DNLE} by using \eqref{smoothing3_DNLE} and we get
\begin{equation}\label{est3u_Benilan_Crandall_type_DNLE}
\|u_{t}(t,x) \|_{1} \leq  C\mu t^{-1}  (t/2)^{-\mu} \leq C t ^{-1-\mu}.
\end{equation}

\subsection{Asymptotic behaviour: Uniform convergence to the asymptotic profile}
 Now we are ready to state the first important result of the paper.
\begin{theorem}\label{ThAsympBehDNLE}
Let $\Omega\subseteq {\mathbb R}^N$ be a bounded domain of class $C^{2,\alpha}$,$\alpha>0$. Let $u(t,\cdot)$ be a weak solution to Problem \ref{DNLE} corresponding to the nonnegative initial datum $u_0\in L^1(\Omega).$
Then there exists a unique self-similar solution of Problem \ref{DNLE}
$$ U(t,x) = t ^{-\mu} f(x), \ t \in (0, +\infty), x \in \Omega,$$
such that
\begin{equation}\label{conv1_DNLE}
 \lim_{t \rightarrow + \infty}   t ^{\mu}|u(t,x) - U(t,x)| =
 \lim_{t \rightarrow + \infty}  | t ^{\mu}u(t,x) - f(x)| = 0,
 \end{equation}
unless $u$ is trivial, $u\equiv0$. The convergence is uniform in space and monotone non-decreasing in time. Moreover, the asymptotic profile $f$ is the unique non-negative solution of the stationary problem:
\begin{equation} \label{eq_f_DNLE}
   \displaystyle{ \Delta_p f^m(x) +\mu f(x)=0,  \ x \in \Omega,} \ \
       f(x)=0, \ x \in \partial \Omega.
  \end{equation}
\end{theorem}

\begin{proof}

\textbf{$1$.  The main tools} are the a-priori estimates \eqref{smoothing3_DNLE} and \eqref{estim_derivative1_DNLE}
\begin{equation}\label{est1u_DNLE}
u(t,x) \leq C t ^{-\mu},
\end{equation}
and
\begin{equation}\label{est2u_DNLE}
u_{t}(t,x) \geq  - C\mu t^{-1}u.
\end{equation}
These estimates make sense also in the classical way since the weak solution $u$ of Problem \ref{DNLE} is in fact locally H\"{o}lder continuous.

\noindent \textbf{$2$. Rescaled orbit and equation.}
We use the rescaling
\begin{equation}\label{scaling_formulas1_DNLE}
u(t,x) = t^{-\mu} v(\tau,x),  \ t=e^{\tau} .
\end{equation}
As we have seen, Problem \ref{DNLE} is mapped into the \emph{rescaled problem}:
\begin{equation}\label{rescaledEq_DNLE}
  \left\{ \begin{array}{ll}
\displaystyle{ v_{\tau}(\tau,x)=\Delta_p v^m(\tau,x) + \mu v(\tau,x)},  &\text{for }\tau\in \mathbb{R} \text{ and } x \in \Omega,  \\
v(0,x)=v_0(x)=u(x,1), &\text{for } x \in \Omega, \\
v(\tau,x)=0, &\text{for }\tau\in \mathbb{R} \text{ and } x \in \partial \Omega.
    \end{array}\right.
\end{equation}
For this problem, we take zero Dirichlet boundary data in the sense that $v(\tau,x)\in W_0^{1,p}(\Omega)$. The initial data are taken non-negative and integrable in $\Omega$. The possibility of delaying the time origin and the regularity theory allow us to assume that $v(x,0)$ is bounded, even continuous.

\noindent \textbf{$3.$ Convergence.}  The advantage of the new variable is seen when we translate the estimate information in terms of $v$. From the a-priori estimates \eqref{est1u_DNLE} and \eqref{est2u_DNLE} we get better estimates for the function $v$:
\begin{equation}\label{est1v_DNLE}
0 \leq v \leq C,
\end{equation}
and
\begin{equation}\label{est2v_DNLE}
v_{\tau}\geq 0.
\end{equation}
We conclude from this that for every $x \in \Omega$ there exists the limit
$$\lim_{\tau \rightarrow \infty}v(\tau,x) = f(x)$$
and this convergence is monotone non-decreasing. Also, from \eqref{est1v_DNLE}, the function $f(x)$ is nontrivial and bounded. Moreover, by (Beppo Levi's) Monotone Convergence Theorem we have
$$ v(\tau,\cdot) \rightarrow f \text{ strong  in }L^1(\Omega).$$ Since there is an $L^{\infty}(\Omega)-$bound the convergence takes place in all $L^q(\Omega)$, $1\leq q\leq p$.
This function $f$ is called an \textbf{asymptotic profile} and we will prove that it is the solution of the stationary elliptic problem associated to the rescaled problem \eqref{rescaledEq_DNLE} and it is unique.

\noindent \textbf{$4.$ Energy estimates. }
We consider the next Lyapunov functional, called \textbf{entropy},
$$E(\tau)=E[v(\tau)]:=\frac{1}{p} \int_{\Omega} | \nabla v^m (\tau,x)|^p dx  -  \frac{m}{m+1}\mu \int_{\Omega}v^{m+1}(\tau,x) dx.$$
We compute the entropy dissipation (Fisher information)
$$
\frac{d}{d\tau}E(\tau) =  - m \int_{\Omega}v^{m-1} (\tau,x) v_{\tau}^2(\tau,x) dx \leq 0,
$$
which means that $E(\tau)$ is a non-increasing function. Then we obtain that
\begin{align}\label{energy1_DNLE}
E(\tau)&= \frac{1}{p} \int_{\Omega} | \nabla v^m(\tau,x)|^p dx  -  \frac{m}{m+1}\mu \int_{\Omega}v^{m+1}(\tau,x) dx  \\
&\leq E(0)=\frac{1}{p} \int_{\Omega} | \nabla v_0^m(x)|^p dx  -  \frac{m}{m+1}\mu \int_{\Omega}v_0^{m+1}(x) dx.
\end{align}
From this energy estimate \eqref{energy1_DNLE} together with a-priori estimate \eqref{est1v_DNLE}
we get that $|\nabla v^m(t,\cdot)|$ is uniformly bounded in time in every $L^q(\Omega)$-norm, $1\leq q\leq p$,
thus weakly convergent up to subsequences. Let us denote by $M$ the uniform bound
for the $L^p(\Omega)$-norm:
\begin{equation}\label{unif_bound_gradient_DNLE}
\int_{\Omega}|\nabla v^m(\tau,x)|^p dx \leq M, \quad \forall t\in \mathbb{R}.
\end{equation}
We deduce that
\begin{equation}\label{weak_conv_gradients1_DNLE}
\partial_{x_i}v(\tau,\cdot) \rightharpoonup \partial_{x_i}f(\cdot) \quad \text{when } \tau \rightarrow \infty
\text{  in  } L^q(\Omega), \quad \text{ for every }1<q\leq p.
\end{equation}

We can also get a uniform bound for $\|v_{\tau}\|_{L^1(\Omega)}$, using the B\'{e}nilan-Crandall type estimate
\eqref{est3u_Benilan_Crandall_type_DNLE} as follows:
$$
\left\|v_{\tau}(\tau,x)  - \mu v(\tau,x) \right\|_{L^1(\Omega)} =  \| \Delta_p v^m(\tau,x) \|_{L^1(\Omega)} = \|\Delta_p \left ( t^\mu u(t,x) \right)^m \|_{L^1(\Omega)} $$
$$=\displaystyle{t^{1+\mu}\|\Delta_p u^m(t,x)  \|_{L^1(\Omega)} =  t^{1+\mu}\|u_{t}(t,x) \|_{L^1(\Omega)} }\leq C_1,$$
and now we obtain a uniform bound in time.
Thus, by now we have:
$$\left\|v_{\tau}  - \mu v \right\|_{L^1(\Omega)} \leq  C_1$$
and then
\begin{equation}\label{est3v_DNLE}
\left\|v_{\tau} \right\|_{L^1(\Omega)} \leq C_2 .
\end{equation}
\textbf{Remark:} in the PLE case these estimates can be obtained also in the $L^p(\Omega)$ because of the contractivity property of the $p$-Laplacian operator in $L^p(\Omega)$.

\noindent \textbf{$5.$ Convergence in measure of gradients.}
The weak convergence of $\nabla v^m $ to $\nabla f^m$  can be improved, and in fact we will prove a stronger result,
the convergence in measure (with respect to the Lebesgue measure $\mathcal{L}$). The idea comes from \cite{MR1354907}.
We refer also to \cite{MR2646119} where the authors prove this result for the fast $p$-Laplacian equation.
A strong argument in our proof are the following inequalities for vectors in $\mathbb{R}^n$.
If $2\leq p$ then
\begin{equation}\label{alg_ineq}
<|a|^{p-2}a-|b|^{p-2}b,a-b> \geq \gamma_1 |a-b|^p , \quad \forall a,b \in \mathbb{R}^N,
\end{equation}
where $\gamma_1=c_p$ is a constant depending on $p$ and $n$.
If $1<p<2$ then
\begin{equation}\label{alg_ineq2}
<|a|^{p-2}a-|b|^{p-2}b,a-b> \geq \gamma_2 \displaystyle{\frac{|a-b|^2}{|a|^{2-p}+|b|^{2-p}}},
\end{equation}
with optimal constant $\gamma_2=c_p=\min\{1,2(p-1)\}.$
For a proof of these inequalities we refer to  \cite{DiBenedetto, MR2646119}.

We prove now that $ \nabla v^m(\tau,\cdot)$ converges in measure to $ \nabla f^m(\cdot)$ when $\tau \rightarrow \infty$.
We remark that it is sufficient to prove that $(\nabla v^m(\tau,\cdot)  )_{\tau>0}$ is Cauchy in measure.
Thus we have to prove that for every $\epsilon>0$ there exists a $T>0$ and $\lambda>0$ such that
$$ \mathcal{L} \left (\{ x \in \Omega : |\nabla v^m(\tau_1,x) - \nabla v^m(\tau_2,x)| > \lambda \}  \right)< \epsilon, \ \forall \tau_1,\tau_2>T.$$
Let $\epsilon>0$ and $S$ be the set whose measure we want to estimate
$$S:=\{ x \in \Omega : |\nabla v^m(\tau_1,x) - \nabla v^m(\tau_2,x)| > \lambda \} .$$
Then
\begin{align*}
S \subset  &\{ x \in \Omega : |\nabla v^m(\tau_1,x)|>A \} \cup  \{ x \in \Omega : |\nabla v^m(\tau_2,x)|>A \} \\
   & \cup \{ x \in \Omega :  |\nabla v^m(\tau_1,x)| \leq A ,  \  |\nabla v^m(\tau_1,x)| \leq A,  \  |\nabla v^m(\tau_1,x) - \nabla v^m(\tau_2,x)| > \lambda \}  \\
   &= S_1 \cup S_2.
\end{align*}
Since $ |\nabla v^m(\tau,\cdot)|$ is uniformly bounded in $L^p(\Omega)$ then $\mathcal{L}(S_1)<\epsilon$ for $\tau_1,\tau_2$ sufficiently large.
Now, in order to estimate $\mathcal{L}(S_2)$, the idea is to use algebraic inequalities \eqref{alg_ineq} and \eqref{alg_ineq2} for the vectors $\nabla v^m(\tau_1)$ and $\nabla v^m(\tau_2)$.

\begin{itemize}
  \item
If $p\geq 2$ then \begin{align*}
S_2 \subset & \{ x \in \Omega :  |\nabla v^m(\tau_1,x)| \leq A ,  \  |\nabla v^m(\tau_1,x)| \leq A,  \\
  & (|\nabla v^m(\tau_1)|^{p-2}\nabla v^m(\tau_1)-|\nabla v^m(\tau_2)|^{p-2}\nabla v^m(\tau_1))\cdot(\nabla v^m(\tau_1)- \nabla v^m(\tau_2)) \geq \gamma_1 \lambda^p := \beta_1\}.
\end{align*}
  \item  If $1<p<2$ then \begin{align*}
S_2 \subset & \{ x \in \Omega :  |\nabla v^m(\tau_1,x)| \leq A ,  \  |\nabla v^m(\tau_1,x)| \leq A,  \\
  & (|\nabla v^m(t_1)|^{p-2}\nabla v^m(\tau_1)-|\nabla v^m(\tau_2)|^{p-2}\nabla v^m(\tau_1))\cdot(\nabla v^m(\tau_1)- \nabla v^m(\tau_2)) \geq \gamma_2 \frac{\lambda^2}{2A^{2-p}} =: \beta_2\}.
\end{align*}
\end{itemize}
We remark that in the particular case of the degenerate PLE we only have to consider the case $p>2$.

Now, for $\beta=\beta_1$ if $p\geq 2$,  respectively $\beta=\beta_2$ if $1<p< 2$, we obtain that
\begin{align*}
\mathcal{L}(S_2) &= \int_{S_2}d\mu \leq \frac{1}{\beta} \int_{\Omega}(|\nabla v^m(\tau_1)|^{p-2}\nabla v^m(\tau_1)-|\nabla v^m(\tau_2)|^{p-2}\nabla v^m(t_1))\cdot(\nabla v^m(\tau_1)- \nabla v^m(\tau_2)) dx \\
&= - \frac{1}{\beta} \int_{\Omega}  \left[ \nabla.  (|\nabla v^m(\tau_1)|^{p-2}\nabla v^m(\tau_1)-|\nabla v^m(\tau_2)|^{p-2}\nabla v^m(\tau_1))\right] \left[v^m(\tau_1)-v^m(\tau_2)\right] dx\\
&= - \frac{1}{\beta} \int_{\Omega} (\Delta_p v^m(\tau_1)  -\Delta_p v^m(\tau_2))(v^m(\tau_1)-v^m(\tau_2)) dx \\
&= - \frac{1}{\beta} \int_{\Omega}\left (v_{\tau}(\tau_1)+ \mu v(\tau_1) - v_{\tau}(\tau_2) - \mu v(\tau_2) \right)(v^m(\tau_1)-v^m(\tau_2)) dx,
\end{align*}
where we used integration by parts. We recall that $v$ is positive and uniformly bounded in time
by \eqref{est1v_DNLE} and the norm $\| v_{\tau}\|_{L^2(\Omega)}$ is also uniformly bounded in time by \eqref{est3v_DNLE}. Thus
$$\mathcal{L}(S_2)< \frac{1}{\beta} C$$
where $C$ is a constant positive number and it follows that
$$\mathcal{L}(S_2)< \epsilon,$$
for $\beta$ big enough.

Thus, we proved that the sequence $(\nabla v^m(\tau,\cdot))_{\tau>0}$ is Cauchy in measure, thus it converges in measure to a
function $W:\Omega \rightarrow \mathbb{R}^N$.
It is well a known fact(Lemma \ref{Lemma_conv_measure} in the Appendix) that if a sequence is uniformly bounded in $L^p(\Omega)$ and converges in measure,
it converges strongly in any $L^q(\Omega) $, for any $1\leq q <p$.
It follows that
$$\nabla v^m(\tau, \cdot) \rightarrow W(\cdot) \quad \text{ in }(L^q(\Omega))^N \text{ when }\tau \rightarrow \infty, \text{ for every } 1\leq q <p.$$
Since we already know the weak convergence \eqref{weak_conv_gradients1_DNLE} we get that $w=\nabla f^m$ a.e. in $\Omega$ and we can conclude that
$$\nabla v^m(\tau, \cdot) \rightarrow \nabla f^m(\cdot) \quad \text{ in measure} $$
and
\begin{equation}\label{conv_Lq_gradients_DNLE}
\nabla v^m(\tau, \cdot) \rightarrow \nabla f^m(\cdot) \quad \text{ in }(L^q(\Omega))^N, \text{ for every } 1\leq q <p.
\end{equation}
Thus, we get that, up to subsequences,
\begin{equation}\label{conv_ae_gradients_DNLE}
\nabla v^m(\tau, \cdot) \rightarrow \nabla f^m(\cdot) \quad \text{a.e. in }\Omega.
\end{equation}

\noindent \textbf{$6.$ The limit is a stationary solution.} Multiply equation \eqref{rescaledEq_DNLE} by any test function $\phi(x) \in C_c^{\infty}(\Omega)$ and integrate in space, $x \in \Omega$, and time between $\tau_1$ and $\tau_2= \tau_1 + T_0$, for a fixed $T_0 >0.$ We get that
\begin{align*}
\int_{\Omega}v(\tau_2) \phi dx - \int_{\Omega}v(\tau_1) \phi dx &= \int_{\tau_1}^{\tau_2}\int_{\Omega} \Delta_p v^m \phi dx dt  + \mu\int_{\tau_1}^{\tau_2}\int_{\Omega}  v \phi dx dt = \\
&=- \int_{\tau_1}^{\tau_2}\int_{\Omega} |\nabla v^m|^{p-2}\nabla v^m  \nabla \phi dx dt  + \mu\int_{\tau_1}^{\tau_2}\int_{\Omega} v \phi dx dt .
\end{align*}
Let $\tau_1\rightarrow \infty$. Then also $\tau_2\rightarrow \infty$ and we get that
$$\int_{\Omega}v(\tau_2) \phi dx - \int_{\Omega}v(\tau_1) \phi dx \rightarrow 0$$
and then, by Lebesgue's Dominated Convergence Theorem, it follows that
$$\mu\int_{\tau_1}^{\tau_2}\int_{\Omega}  v \phi dx dt \rightarrow T_0 \mu \int_{\Omega}  f \phi dx.$$
Now, we need to obtain the convergence of the integrals involving gradients:
\begin{equation}\label{weak_conv_v_f_DNLE}
\int_{\tau_1}^{\tau_2} \int_{\Omega}|\nabla v^m|^{p-2}\nabla v^m  \nabla \phi dxdt  \rightarrow T_0\int_{\Omega}|\nabla f^m|^{p-2}\nabla f^m  \nabla \phi dx
\end{equation}
when $\tau_1 \rightarrow \infty$ in order to obtain that
$$0=- T_0 \int_{\Omega} |\nabla f^m|^{p-2}\nabla f^m  \nabla \phi dx   + \mu T_0 \int f \phi dx ,$$
and dividing by $T_0$ and integrating by parts we get
$$0= \int \Delta_p f^m  \phi dx   + \mu\int f \phi dx ,$$
which proves that $f$ is a weak solution of the stationary problem \eqref{eq_f_DNLE}
\begin{equation*}
   \displaystyle{ -\Delta_p f^m(x) =\mu f(x),  \ x \in \Omega.}
  \end{equation*}

\vspace{1cm}
In order to justify assertion \eqref{weak_conv_v_f_DNLE}, we remark that, after a change of variables, the left term can be written as
$$
\int_{0}^{T_0} \int_{\Omega}|\nabla v^m(\tau+\tau_1,x)|^{p-2}\nabla v^m(\tau+\tau_1,x)  \nabla \phi dxdt.  $$
Thus, it is enough to prove that
\begin{equation*}
\int_{0}^{T_0} \int_{\Omega}|\nabla v^m(\tau+n,x)|^{p-2}\nabla v^m(\tau+n,x)  \nabla \phi dxdt  \rightarrow T_0\int_{\Omega}|\nabla f^m|^{p-2}\nabla f^m  \nabla \phi dx  \text{ when }n \rightarrow \infty.
\end{equation*}
The idea will be to use Lemma \ref{brezis} from the Appendix in the following context.
Let $\Omega_1= [0,T_0 ) \times \Omega $ (finite measure space), $H= \mathbb{R}^N $ (Hilbert space) and let
$$A: H \rightarrow H, \quad A(Z)= |Z|^{p-2}Z.$$
Then $A$ is single valued, monotone and $R(A+I)=H$ and therefore, according to the theory of monotone operators, $A$ is maximal monotone.
We consider the sequences
$$Z_n(t,x)= \nabla v^m(\tau+n,x): \Omega_1 \rightarrow  H,
$$
$$
W_n(\tau,x)=A(Z_n(\tau,x))= |\nabla v^m(\tau+n,x)|^{p-2}\nabla v^m(\tau+n,x): \Omega_1 \rightarrow  H.
$$
The hypothesis of Lemma \ref{brezis} are satisfied as follows. From \eqref{conv_ae_gradients_DNLE}
$$
Z_n(\tau,x) \rightarrow Z(\tau,x)= \nabla f^m(x) \text{ a.e. on  } \Omega_1.
$$
Let $q= \frac{p}{p-1}$.
Then $W_n(t,x) $ is uniformly bounded in $L^q (\Omega_1;H)$, by the energy estimate \eqref{unif_bound_gradient_DNLE}, since
$$
 \int_{0}^{T_0} \int_{\Omega}||\nabla v^m(\tau+n,x)|^{p-2}\nabla v^m(\tau+n,x)  |^q dx dt \leq MT_0 , \ \forall n>0,
 $$
and thus it converges weakly (up to subsequences) to a function $W$ in $L^q (\Omega_1;H)$ when $n \rightarrow \infty.$
But since $\Omega_1$ is bounded, then weak convergence in $L^q (\Omega_1;H)$ implies the weak convergence in $L^1 (\Omega_1;H)$
and thus
$$ W_n(\tau,x)\rightharpoonup W(\tau,x) \text{ weakly in }L^1 (\Omega_1;H).$$

Now, according to Lemma \eqref{brezis}, we obtain that $$W(\tau,x)=A(Z(\tau,x))=|\nabla f^m(x)|^{p-2}\nabla f^m(x).$$
Thus, the weak limit of $W_n(\tau,x)$ is unique and we get that
$$|\nabla v^m(\tau+n,x)|^{p-2}\nabla v^m(\tau+n,x) \rightharpoonup |\nabla f^m(x)|^{p-2}\nabla f^m(x)\text{ weakly in }L^q (\Omega_1;H).$$
By taking $\phi$ smooth and compactly supported in $\Omega$ we obtain the desired convergence
\begin{align*}
\int_{0}^{T_0} \int_{\Omega}|\nabla v^m(\tau+n,x)|^{p-2}\nabla v^m(\tau+n,x)  \nabla \phi dxdt  &\rightarrow \int_{0}^{T_0} \int_{\Omega}|\nabla f^m|^{p-2}\nabla f^m \nabla \phi dx dt \\
&= T_0 \int_{\Omega}|\nabla f^m|^{p-2}\nabla f^m  \nabla \phi dx.
\end{align*}

\noindent \textbf{$7.$ Uniqueness of the stationary solution.} Let us prove that the nonnegative and nontrivial stationary solution is unique. If we have two stationary solutions of \eqref{eq_f_DNLE}, $f_1$ and $f_2$, we can construct solutions of the DNLE of the form
$$ U_1(t,x)= t^{-\mu}f_1(x), \quad U_2(t,x)= (t+s)^{-\mu}f_2(x),$$
for some $s>0$. $U_2$ has initial data $U_2(x,0)= s^{-\mu}f_2(x)$. Formally, $U_1(x,0)$ has infinite values and then by the Comparison Principle we conclude that $U_2(t,x) \leq U_1(t,x)$. The technical details of the proof are as follows: by the $L^1-$dependence theorem of weak solutions of Problem \ref{DNLE} we know that
$$\int_{\Omega}(U_2(t,x)-U_1(t,x))_+dx,$$
is decreasing in time. The proof of this fact is standard: we perform the difference of the equations of $U_1$ and $U_2$,
multiplying by $h(w)$, where $w=U_1^m-U_2^m$ and $h$ is a $C^1(\mathbb{R})$ function such that $0\leq h\leq1 , \ h(s)=0$
 for $s\leq 0$ and $h'(s)>0$ for $s\geq0$, and then integrate on $\Omega$.
Notice that
 $$0 \leq \mu_0:= \inf_{x \in \Omega}h'(w(x))<\infty.$$
Because of the nonlinearity of the $p-$laplacian operator, we will use again the algebraic inequalities \eqref{alg_ineq} and \eqref{alg_ineq2}.
\begin{align*}
\int_{\Omega}(U_2(t,x)-U_1(t,x))_t \ h(w)dx &= \int_{\Omega}(\Delta_p U_2^m(t,x)-\Delta_p U_1^m(t,x))h(w)dx \\
&= - \int_{\Omega}\left (|\nabla U_1^m|^{p-2}\nabla U_1^m - |\nabla U_2^m|^{p-2}\nabla U_2^m \right) \nabla h(w) dx\\
&= - \int_{\Omega}\left (|\nabla U_1^m|^{p-2}\nabla U_1^m - |\nabla U_2^m|^{p-2}\nabla U_2^m \right) h'(w) \nabla(U_1^m-U_2^m) dx\\
&\leq -   \mu_0\int_{\Omega}\left (|\nabla U_1^m|^{p-2}\nabla U_1^m - |\nabla U_2^m|^{p-2}\nabla U_2^m \right) \nabla(U_1^m-U_2^m) dx.
\end{align*}
Now, if $p\geq 2$ we obtain that
$$\int_{\Omega}(U_2(t,x)-U_1(t,x))_t \ h(w)dx \leq - \mu_0 \gamma_1 \int_{\Omega} |\nabla U_1^m - \nabla U_2^m|^p dx  \leq 0,$$
and if $1<p<2$ the estimate will be
$$\int_{\Omega}(U_2(t,x)-U_1(t,x))_t \ h(w)dx \leq - \mu_0 \gamma_2 \int_{\Omega} \frac{|\nabla U_1^m - \nabla U_2^m|^2 }{|\nabla U_1^m|^{2-p}+|\nabla U_2^m|^{2-p}}dx \leq 0.$$
Letting $h$ converge to the sign function $\text{sign}_0^+$ we get that
$$\frac{d }{dt} \int_{\Omega}(U_2(t,x)-U_1(t,x))_+dx \leq 0.$$
Now, the integral goes to $0$ as $t\rightarrow 0$ because $U_1(t,x)$ goes pointwise to infinity as $t\rightarrow 0$ and then $U_2(t,x) >U_1(t,x)$ for $t$ large enough. We conclude that $(U_2(t,x)-U_1(t,x))_+=0$ a.e. in $x$ for every $t>0$. Thus $U_2(t,x)\leq U_1(t,x)$ a.e. in $x$ for every $t>0$. Using the form of $U_1$ and $U_2$, we get
$$f_2(x) \leq \left(\frac{t+s}{t}\right)^{\mu}f_1(x).$$
Letting $s\rightarrow 0$ we get $f_2(x)\leq f_1(x).$ The converse inequality is similar.

\noindent \textbf{$8.$ Better convergence.} We have established the result \eqref{conv1_DNLE} in the sense of $L^1(\Omega)$ convergence.
The passage to uniform convergence depends on having better regularity for the solutions , i.e. on a compactness argument.
As we mentioned before, \eqref{regularity_DNLE}, uniformly bounded solutions of the DNLE are $C^{\alpha}$ continuous in
space and time with uniform H\"{o}lder exponent and coefficients.

Consider now the second type of rescaling that we may call \emph{fixed-rate rescaling}
\begin{equation}\label{fixed_rate_rescaling_DNLE}
u_{\lambda}(t,x)=\lambda^{\mu}u(\lambda t,x).
\end{equation}
For every $\lambda>0$ the function $u_{\lambda}$ is still a solution of the DNLE
 to which the a-priori estimate \eqref{est1u_DNLE} applies. Hence, in a set of the form $(1,2)\times \Omega$ this
 family is equi-continuous and by Ascoli-Arzela Theorem it converges along a subsequence $\lambda_j \rightarrow \infty$.
 Now, observe that
 $$ u_{\lambda}(1,x)= v(\log \lambda,x)$$
 to conclude that $v(\log \lambda_j,x)$ converges uniformly. Since the limit is fixed, $f$,
 the whole family $v(\tau,x)$ converges as $\tau \rightarrow \infty$ and \eqref{conv1_DNLE} is proved.
\end{proof}

\subsection{Some remarks on the asymptotic profile}

\noindent \textbf{$1.$ Existence.} The proof of Theorem \ref{ThAsympBehDNLE} also guarantees the existence of a solution of the stationary problem \eqref{eq_f_DNLE} by obtaining it as the limit of $v(t,\cdot)$ when $t$ goes to $\infty$. As we have previously established, this solution is called the asymptotic profile of parabolic problem \eqref{DNLE}.

Furthermore, we recall a second proof of existence, based on an \emph{entropy method}, which can be applied also for the general case of solutions with changing sign.

The elliptic problem \eqref{eq_f_DNLE} can be written in terms of the function $w=f^m$ (the notation makes sense since $f>0$ in $\Omega$ by Maximum Principle) as
\begin{equation} \label{eq_f_DNLE_2}
   \displaystyle{ \Delta_p w(x) +\mu w^{\frac{1}{m}}(x)=0,  \ x \in \Omega,} \ \
       w(x)=0, \ x \in \partial \Omega.
  \end{equation}

The typical approach to solving equation \eqref{eq_f_DNLE_2} for the experts in elliptic equations is to view the solution $w$ as a critical point of the functional
\begin{equation}\label{functional_DNLE}
J(g)= \frac{1}{p} \int_{\Omega} | \nabla w|^p dx  -  \frac{m}{m+1}\mu \int_{\Omega}w^{\frac{m+1}{m}} dx.
\end{equation}
The proof is classical and we resume it following the ideas from \cite{JLVsurvey}.
It can be showed that $J$ is well defined in $W_0^{1,p}(\Omega)$ since $m(p-1)>1$, $J$ is bounded from below via Poincare's Inequality, and also the infimum is negative. Moreover, along any minimizing sequence there is convergence in $W_0^{1,p}(\Omega)$ and the infimum is taken, hence $J$ has a minimum.
Also,
$$J(g)\geq J(w), \quad \forall g  \in W_0^{1,p}(\Omega),$$
where $w$ is the solution of \eqref{eq_f_DNLE_2} and it follows that $w$ is the point where $J$ attains its minimum.

\medskip

\noindent \textbf{$2$. Uniqueness. }We already proved the uniqueness of the asymptotic profile in point $7$ of the proof of Theorem \ref{ThAsympBehDNLE}.

\medskip

\noindent \textbf{$3$. Regularity. }We know that $w$ is a bounded solution of equation $$\Delta_p w(x) +\mu w^{\frac{1}{m}}(x)=0.$$
By known regularity results for degenerate elliptic equations, we get that $w \in C^{1,\beta}(\overline{\Omega})$ for some $\beta \in (0,1]$.

\medskip

\noindent \textbf{$4$. Behaviour near the boundary. }Concerning the behaviour of $f$, the following estimates were proved in \cite{ManfrediVespri}:
\begin{equation}\label{behaviour_gradient_f_DNLE}
|\nabla f|\leq C_0d(x)^{1/m -1}, \quad \forall x \in \Omega,
\end{equation}
and
\begin{equation}\label{behaviour_f_DNLE}
C_1 d(x)^{1/m} \leq f(x) \leq C_2 d(x)^{1/m}, \quad \forall x \in \Omega.
\end{equation}
that is $w(x)=f^m(x)$ has a linear growth near the boundary.

Also, $w$ satisfies a Boundary Principle
\begin{equation}\label{bdry_principle}
\frac{\partial w}{\partial \nu}(x)\equiv \nabla w(x) \cdot \nu(x) <0, \quad \text{for } x \in \partial \Omega,
\end{equation}
where $\nu(x)$ denotes the outward unit normal vector to $\partial \Omega$ at the point $x$. Notice that, in the case of the PLE, the boundary principle is satisfied by $f$.

\begin{section}{Rate of Convergence for $m(p-1)>1$. Proof of Theorem \ref{Th_rate_conv_DNLE}}\label{Section_Rate_DNLE}

\noindent   In this section we will give the proof of Theorem \ref{Th_rate_conv_DNLE}. We previously showed in Theorem \ref{ThAsympBehDNLE} that for any initial datum $u_0 \in L^1(\Omega)$, the corresponding rescaled solution $t^{\mu}u(t,x)$ converges to an unique asymptotic profile $f$ uniformly in space and monotone nondecreasing in time. The goal of this section is to provide sharp convergence rates, namely to prove that
\begin{equation}\label{eq_rate_conv_DNLE}
\left|(1+t)^{\mu} u(t,x)-f(x)\right| \leq \mathcal{C}f(x)(1+t)^{-1} \quad \text{for all }t\geq t_0,\ x \in\overline{\Omega},
\end{equation}
where $f$ is the solution of the elliptic problem \eqref{eq_f_DNLE}
\begin{equation*}\label{eq_f_DNLE_3}
   \displaystyle{ \Delta_p f^m + \mu f=0,  \ x \in \Omega, \quad  f=0 \text{  on  } \partial \Omega.}
    \end{equation*}

The proof of this result is based on the techniques introduced by Aronson and Peletier for the PME in \cite{Aronson1981378}. Although a similar proof can be adapted to the case of the DNLE with some lengthly arguments, in this work we will give a simpler proof based on the results of Section \ref{SectionAsympBehDNLE}. Let us explain the strategy of the proof.

\noindent \textbf{1. Improved upper bound.} In Theorem \ref{Th_upper_bound} we prove that there exists a constant $s_1>0$ depending only on $p,m,N,u_0$ and $\Omega$ such that
   $$0\leq u(t,x) \leq (s_1+t)^{-\mu}f(x), \quad \forall x\in \Omega, \ t \geq 1.$$

\noindent \textbf{2. Positivity.} In Proposition \ref{pro_posit_boundary} we prove that even if $u_0$ has compact support there exists $T'>0$ depending only on $p,m,N,u_0$ and $\Omega$ such that
   $$u(t,x)>0, \quad \forall x\in \Omega,\ t >T'.$$

\noindent \textbf{3. Sharp lower bound.} In Theorem \ref{Th_lower_bound} we prove that there exist $T''\geq0$ and $s_0>0$ depending only on $p,m,N,u_0$ and $\Omega$ such that
  $$u(t,x) \geq (s_0+t)^{-\mu}f(x), \quad \forall x \in \overline{\Omega},\ t\geq T''.$$

Then, the estimate \eqref{eq_rate_conv_DNLE} follows as a consequence of the upper and lower bounds together with the boundedness of the asymptotic profile $0 \leq f \leq C$.

\medskip

\subsection{Reduction}
Let $\Omega\subseteq {\mathbb R}^N$ be a  a bounded domain of class $C^{2,\alpha}$, $\alpha>0$. Let $u(t,\cdot)$ be a weak solution to the Problem \ref{DNLE} corresponding to the nonnegative initial datum $u_0\in L^1(\Omega).$
Previously, in Theorem  \ref{ThAsympBehDNLE} we proved that the rescaled solution $ t^{\mu}u(t,x)$ is monotone increasing in time and convergent to the function $f(x)$, thus bounded from above by $f(x)$:
\begin{equation}\label{upper_estim_u_f}
u(t,x) \leq t^{-\mu}f(x),\quad \forall t\in (0,\infty), \ x\in \overline{\Omega}  .
\end{equation}
In virtue of this result we can assume that the data satisfy the following conditions, denoted as \textbf{Hypothesis (H)}:\\
$\textbf{(H1)}$ $\Omega\subseteq {\mathbb R}^N$  be a  a bounded domain of class $C^{2,\alpha}$, $\alpha>0$.\\
$\textbf{(H2)}$ $u_0$ is a nonnegative function defined on $\overline{\Omega}$ such that $u_0\in L^1(\Omega)$, $u_0=0$ on $ \partial \Omega$ and there exists $s_1>0$ such that
\begin{equation}\label{ineq_u0_f}
u_0(x) \leq s_1 f(x), \quad \forall x \in \overline{\Omega}. \end{equation}

\noindent Since the DNLE is invariant under time displacement then (H2) is satisfied by starting with initial data $u(t_0, \cdot)$, where $t_0>0$. We assume henceforth that such a displacement in time has been done.

\subsection{Improved upper bound for $u$}

In the following theorem, we will improve the upper bound \eqref{upper_estim_u_f} of $u$ that we have previously proved in Theorem \ref{ThAsympBehDNLE}.
\medskip

\begin{theorem}\label{Th_upper_bound} \textbf{(Quantitative upper bound)}
Assume that $\Omega$ and $u_0$ satisfy the hypothesis (H), and let $u$ be the corresponding solution of the Problem \ref{DNLE}.
Then there exists a constant $s_1>0$ such that
\begin{equation}\label{upper_bound}
u(t,x) \leq (s_1+t)^{-\mu}f(x) \quad \forall t\geq 0, \ x\in \overline{\Omega}.
\end{equation}
where $s_1$ depends on $p,m,N,u_0$ and $\Omega$.
\end{theorem}
\begin{proof}

The proof of \eqref{upper_bound} relies on the Comparison Principle for the DNLE.
Consider as comparison function the separate variable solution $U$ of the DNLE given by
$$U(t,x):=U(t,x;s_1)=(s_1 +t )^{-\mu}f(x),$$
where $s_1>0$  is a constant given in the Hypothesis (H2) which satisfies  the inequality \eqref{ineq_u0_f}
\begin{equation*}\label{inequality_tau1_DNLE}
u_0(x) \leq U(0,x)=s_1 f(x)   , \quad \forall x \in  \overline{\Omega}.
\end{equation*}
Therefore by comparison, it follows that
$$u(t,x) \leq U(t,x), \quad \forall t\geq 0, \ x\in \overline{\Omega}.$$

\end{proof}

\subsection{Positivity of $u$}
In this subsection we prove the positivity of $u$ in $\Omega$, under the  Hypothesis (H). This will be done in two steps. First, in Proposition \ref{pro_interior_positivity},
 we will prove the positivity of $u$ in a domain $\Omega_{I,\delta} \subset \Omega$ and then we complete the result by proving positivity of $u$ up to the boundary in Proposition \ref{pro_posit_boundary}.
In this direction, we will make use of the properties of the distance to the boundary function $d(x)=d(x,\partial \Omega)$ stated in Subsection \ref{subsection_dist_bdry} of the Appendix. In terms of $d(x)$ we define the following sets
$$\Omega_{I,r} = \{x \in \overline{\Omega}:d(x)> r\}, \quad \Omega_{r}=\Omega \setminus \overline{\Omega_{I,r}}= \{x \in \overline{\Omega}:d(x)< r\}.$$

\medskip

\begin{pro}\label{pro_interior_positivity}  \textbf{(Inner positivity)}
Assume that $\Omega$ and $u_0$ satisfy $(H)$ and let $u$ denote the weak solution of Problem \ref{DNLE} and  $f$ denote the solution of problem \eqref{eq_f_DNLE}. Let $0< 2 \delta < \xi_0$ fixed, where $\xi_0$ is defined in Lemma \ref{properties_d(x)}.
Then there exist $\epsilon>0$ and $T_1\in [0,+\infty)$ such that
$$u(T_1,x)>\epsilon \text{ for all } x \in \Omega_{I,\delta},$$
where $\epsilon$ and $T_1$ depend only on $m,p,N,\Omega$ and $u_0$.
\end{pro}
\begin{proof}
The main tools are the uniform convergence (Theorem \ref{ThAsympBehDNLE}) of the rescaled solution $$v(\tau, x)=t^{\mu}u(t,x), \quad t=e^{\tau},$$ defined in \eqref{scaling_formulas1_DNLE} to the asymptotic profile $f$ and the properties of $f$ given by \eqref{behaviour_f_DNLE}. More exactly, there exists $C_1,C_2>0$ such that
\begin{equation}\label{behaviour_f_2}
C_1 d^{\frac{1}{m}}(x) \leq f(x)\leq C_2 d^{\frac{1}{m}}(x), \ \forall x \in \Omega.
\end{equation}
Let $\epsilon_0=\frac{C_1}{2} \delta^{\frac{1}{m}} > 0$. Then there exists $T_0 \geq 0$ such that
$$\|v(\tau,x) - f(x)\|_{L^{\infty}(\Omega)} \leq \epsilon_0 ,\quad \forall \tau\geq T_0.$$
Since the convergence of $v(\tau,x)$ to $f(x)$ is monotone nondecreasing in $\tau$ we derive that
$$v(\tau,x) \geq f(x)-\epsilon_0, \quad \forall x \in \Omega, \ \forall \tau\geq T_0,$$
and then, by using \eqref{behaviour_f_2}, for $x \in\Omega_{I,\delta}$ we obtain the lower bound
$$v(\tau,x) \geq C_1 d^{\frac{1}{m}}(x) - \epsilon_0  \geq C_1 \delta^{\frac{1}{m}} - \epsilon_0 =\epsilon_0.$$
In terms of $u(t,x)$ these estimates rewrites as
$$u(t,x)\geq \epsilon_0 t^{-\mu}, \quad \forall x \in\Omega_{I, \delta}, \ \forall t\geq e^{T_0}.$$
Let
\begin{equation}\label{constants_1}
T_1=e^{T_0}, \quad \epsilon=\epsilon_0 T_1^{-\mu}=\epsilon_0 e^{-T_0\mu}.
\end{equation}
Then
$$u(T_1,x) \geq \epsilon, \quad \forall x \in\Omega_{I,\delta}.$$

\end{proof}

\begin{pro}\label{pro_posit_boundary} \textbf{(Positivity up to the boundary)}
Assume that $\Omega$ and $u_0$ satisfy the hypothesis $(H)$ and let $u$ denote the weak solution of the problem \ref{DNLE} and  $f$ denote the solution of the problem \eqref{eq_f_DNLE}. Consider $T_1>0$ given by Proposition \ref{pro_interior_positivity}.
Then there exists $T_2>0$
$$u(T_1+T_2,x) >0, \quad \forall x \in \Omega,$$
\end{pro}
\noindent and $T_2$ depends only on $m,p,N,\Omega$ and $u_0$.
\begin{proof}
We consider $0<2\delta<\xi_0$ as in Proposition \ref{pro_interior_positivity}.
First, by \eqref{boundary_property}, we observe that
$$\Omega_{2\delta} \subset \bigcup_{ \{y \in \partial \Omega_{I,2\delta} \}} B_{2\delta}(y). $$
Then, since we have already proved the positivity inside the domain in Proposition \ref{pro_interior_positivity}, it is enough to demonstrate that there
exists $T_2\geq0$ such that
$$u(T_1+T_2,x)>0, \quad \forall x\in B_{2\delta}(y),\ \forall y \in \partial \Omega_{I, 2\delta}.$$
Let $\epsilon$ given by \eqref{constants_1}.
Let $y \in \partial \Omega_{I,2\delta}$ and consider the Barenblatt solution $\mathcal{U}$ as in Section \ref{Sim_Sol_Subsection} such that
$$\text{supp }\mathcal{U}(x-y,0;a,s)=  \overline{B}_{\delta}(y) \  \text{ and } \ \text{max } \mathcal{U}(x-y,0;a,s)= \epsilon,$$
that is
\begin{equation}\label{constants_2}
a= \left(\frac{\epsilon}{c}\delta^N\right)^{\beta(m(p-1)-1)}, \quad s= \delta^p \left(\frac{c}{\epsilon}\right)^{m(p-1)-1}.
\end{equation}
Now, assume the time $T_2=t$ when $\text{supp }\mathcal{U}(x-y,t;a,s)$ riches the boundary of $ \Omega$, that is when $$\text{supp }\mathcal{U}(x-y,t;a,s)= \overline{B}_{2 \delta}(y)$$
which implies
\begin{equation}\label{constants_3}T_2=\left(\frac{2\delta}{a}\right)^{1/\beta}-s.\end{equation}
We want to apply now the Parabolic Comparison Principle. To this aim we need to compare $u(T_1+t,x)$ and $\mathcal{U}(x-y,t;a,s)$ when $(t,x)$ belongs to the parabolic boundary $\{0\}\times B_{2\delta} \cup [0,T_2]\times \partial B_{2\delta}$. Firstly, for $t=0$ and $x \in B_{2\delta}$, we have

\[u(T_1,x) \geq
\left\{
  \begin{array}{ll}
    \epsilon \geq \mathcal{U}(x-y,0;a,\tau), & \hbox{$x \in B_{\delta}(y)$;} \\
   0 = \mathcal{U}(x-y,0;a,\tau), & \hbox{$x \in B_{2\delta}(y) \setminus B_{\delta}(y)$.}
  \end{array}
\right.
\]

Secondly, when $t\in[0,T_2]$ and $x \in \partial B_{2\delta}$ we have
\[u(T_1+t,x) \geq
\left\{
  \begin{array}{ll}
    \epsilon \geq \mathcal{U}(x-y,0;a,\tau), & \hbox{$x \in \partial B_{2\delta} \setminus \partial \Omega$;} \\
   0 = \mathcal{U}(x-y,0;a,\tau), & \hbox{$x \in \partial B_{2\delta} \cap \partial \Omega$.}
  \end{array}
\right.
\]
Therefore we obtain that
$$u(T_1+T_2,x)= \mathcal{U}(x-y,T_1+T_2;a,s), \quad \forall x\in B_{2\delta}(y).$$
We notice from \eqref{constants_2} and \eqref{constants_3} that $T_2$ does not depend on the point $y$, but only on the data.

\end{proof}

\subsection{Sharp lower bound for $u$}
In the following theorem we will derive a lower bound for $u$, similar to the upper bound we have previously proved in Theorem \ref{Th_upper_bound}.
\medskip

\begin{theorem} (\textbf{Quantitative lower bound})\label{Th_lower_bound}
Assume that $\Omega$ and $u_0$ satisfy the hypothesis $(H)$ and let $u$ be the weak solution of the problem \ref{DNLE} and  $f$ be the solution of the problem \eqref{eq_f_DNLE}. Then there exist two positive constants $s_0>0$ and $T_4>0$ such that
\begin{equation}\label{ineq12}
u(t,x) \geq (s_0 +t)^{-\mu}f(x), \quad \forall x \in \overline{\Omega} ,\ \forall t \in [T_4,+\infty),
\end{equation}
where $s_0$ and $T_4$ depend only on $m,p,N,\Omega$ and $u_0$.
\end{theorem}

\noindent Before we start the proof of this theorem, we will establish the following preliminary results.

\noindent Let
\[
\mathcal{V}(x-y,t;M,s)=(t+s)^{-\alpha}[g^{\frac{1}{m}}(|x-y|(t+s)^{-\beta})]_+
\]
be the so called intermediate selfsimilar solutions defined in Section \ref{Sim_Sol_Subsection}.

\medskip

\begin{pro}\label{Pro_comp_InterSelfSim}
Under the assumptions of Theorem \ref{Th_lower_bound}, for every  $y \in \partial \Omega_{I,2\delta}$, where $0<2\delta<\xi_0$ is fixed as in Proposition \ref{pro_interior_positivity}, we can choose constants $M$ and $s$ such that there exists a time $T_3>0$, for which the next inequality holds:
\begin{equation}\label{ineq8}
u(T_1+t,x) \geq \mathcal{V}(x-y,t;M,s) , \quad  \forall x \in \overline{\Omega},\ \forall t \in [0,T_3],\end{equation}
where $T_1$ is the time we have obtained in Proposition \ref{pro_interior_positivity} and $T_3>0$ depends only on $m,p,N,\Omega$ and $u_0$ and it is independent of $y$.
\end{pro}
\begin{proof}
The same ideas as in Proposition \ref{pro_posit_boundary} apply since the functions $\mathcal{V}$  have a similar behaviour as the Barenblatt functions. Consider the selfsimilar subsolution $\mathcal{V}$ such that
$$\text{supp }\mathcal{V}(x-y,0;M,s)=  \overline{B}_{\delta}(y) \  \text{ and } \ \text{max } \mathcal{V}(x-y,0;M,s)= \epsilon,$$
that is
\begin{align*}\label{constants_4}
M=\epsilon s^{\alpha},  \quad \delta= s^{\beta}a(M)=s^{\beta}M^{\frac{m(p-1)-1}{p}}a(1) \quad
\text{and} \quad  s= \left (\frac{\delta}{a(1)}\right)^p \epsilon^{-(m(p-1)-1)} .
\end{align*}
The exact values of $M$ and $s$ are not important; what matters is that they depend only on the data and, in particular,
they are independent of $y$.

Now, consider the time $T_3=t$ when $\text{supp }\mathcal{V}(x-y,t;M,s)$ riches the boundary of $ \Omega$, that is when
\begin{equation}\label{support_V}
\text{supp }\mathcal{V}(x-y,t;M,s)= \overline{B}_{2 \delta}(y)
\end{equation}
from where we deduce the explicit value for $T_3$
\begin{equation}\label{constants_5}
T_3=\left(\frac{2\delta}{a}\right)^{1/\beta}-s.
\end{equation}
Then the Parabolic Comparison Principle can be applied to $u$ and $\mathcal{V}$ on the parabolic domain $[T_1,T_1+T_3] \times \overline{\Omega}$ as in Proposition \ref{pro_posit_boundary} and we obtain that
\begin{equation*}\label{ineq88}
u(T_1+t)\geq \mathcal{V}(x-y,t;M,s), \quad \forall x\in \Omega, \ \forall t \in[0,T_3].
\end{equation*}
We notice from \eqref{constants_4} and \eqref{constants_5} that $T_3$ does not depend on the point $y$, but only on the data.
\end{proof}

\noindent We define
\begin{equation}\label{T4}
T_4=T_1+T_3
\end{equation}
where $T_1$ and $T_4$ are given by Proposition \ref{pro_interior_positivity}, respectively Proposition \ref{Pro_comp_InterSelfSim}.
\medskip

\begin{pro}\textbf{(Boundary behaviour)}
Under the hypothesis of Theorem \ref{Th_lower_bound}, let $T_4$ as in \eqref{T4} and $\delta$ as in Proposition \ref{pro_interior_positivity}.
Then there exists a constant $\omega >0$
$$u^m(T_4,x) \geq \omega d(x) \quad \text{for } x \in \Omega_{\delta}.$$
\end{pro}
\noindent such that $\omega$ depends only on $m,p,N,\Omega$ and $u_0$.
\begin{proof}

Let $y \in \partial \Omega_{I,2 \delta}$ . Then by Lemma \eqref{properties_d(x)}, there exists a unique $z(y) \in \partial \Omega$ such that $\delta=d(y)=|y-z(y)|$.
Let
$$
\mathcal{V}(x-y,T_3;M,s)=(T_3+s)^{-\alpha}[g^{\frac{1}{m}}(|x-y|(T_3+s)^{-\beta})]_+
$$
be the self-similar subsolution obtained in Proposition \ref{Pro_comp_InterSelfSim} where $M$ and $s$ are given by formulas \eqref{constants_4}. Let $[0,a)$ the largest interval starting from $0$ where $g>0$ and by \eqref{support_V} it follows that  $a=a(M)=2\delta(T_3+s)^{-\beta} $.
Then, in view of the continuity of $g'$, there exist $k_0<k_1<0$ such that
$$ k_0\leq g'(\eta)\leq k_1, \quad  \forall \eta \in [\delta(T_3+s)^{-\beta}  ,2\delta(T_3+s)^{-\beta} ],$$
and it follows that
\begin{equation}\label{ineq11}
g(\eta) \geq |k_1|(a-\eta), \quad \forall \eta \in  [\delta(T_3+s)^{-\beta}  ,2\delta(T_3+s)^{-\beta} ].
\end{equation}
Thus it follows from \eqref{ineq8} and \eqref{ineq11} that for every $x \in \Omega$ on the segment between $y$ and $z(y)$ such that $\delta<|x-y| <2\delta$, that is $d(x)<\delta$, $u$  can be bounded from bellow as follows
\begin{align}\label{ineq9}
u^m(T_4,x) &\geq \mathcal{V}^m((x-y,T_3;M,s)= (T_3+s)^{-\alpha m}g(|x-y|(T_3+s)^{-\beta})  \nonumber \\
&\geq |k_1|(T_3+s)^{-\alpha m} \left(2\delta(T_3+s)^{-\beta}  - |x-y|(T_3+s)^{-\beta}\right)  \nonumber\\
&= |k_1|(T_3+s)^{-\alpha m - \beta} (2\delta -|x-y|) \nonumber \\
&=\omega d(x),
\end{align}
where $\omega=|k_1|(T_3+s)^{-\alpha m- \beta}  \in \mathbb{R}^+$ is a constant which depends on the data, but not on $y$ and $x$.
Observe that \eqref{ineq9} holds for arbitrary $y \in \partial \Omega_{I,2 \delta}$ and for all $x$  on the inward directed normal through $y$ provided that $d(x)\leq \delta$. As we remarked in Lemma \eqref{properties_d(x)}, the normal map $H_r$ is a homeomorphism for all $r \in [0,\xi_0)$.
Therefore, it follows from \eqref{ineq9} that
\begin{equation}\label{ineq10}
u^m(T_4,x) \geq \omega d(x) \quad \text{for } x \in  \Omega_{\delta}.
\end{equation}

\end{proof}

\noindent \textbf{Proof of Theorem \ref{Th_lower_bound}.}
First we will prove there exists $k_2>0$ such that
\begin{equation}\label{lower_comp_f_2}
u(T_4,x) \geq k_2 f(x) , \quad \forall x \in \Omega.
\end{equation}
By Proposition \ref{Pro_comp_InterSelfSim} and Theorem \ref{behaviour_f_DNLE} $u$ satisfies
$$u^m(T_4,x) \geq \omega d(x)\geq \omega C_2^m f^m(x),    \quad \text{for } x \in  \Omega_{\delta}.$$
Moreover, by Proposition \ref{pro_interior_positivity} and the boundedness of the profile $0 \leq f \leq C$ in $\Omega$ we obtain that
$$u(T_1,x) \geq \epsilon \geq \frac{\epsilon}{C}f(x), \quad \forall x \in\Omega_{I,\delta}.$$
Then, since $T_1<T_4$,  inequality \eqref{lower_comp_f_2} is satisfied with $\displaystyle{k_2= \min\{\epsilon/C, \omega^{1/m}C_2\}.}$

Finally, let $U(t,x)=(s_0 +t)^{-\mu}f(x)$ be the separable solution of the DNLE with initial data $s_0^{-\mu}f$, where $s_0$ is defined by the relation $s_0^{-\mu}=k_2$. Then
$$u(T_4,x) \geq U(0,x), \quad \forall x \in \Omega,$$
and \eqref{ineq12} follows by Comparison Principle.

\noindent \textbf{End of the proof of Theorem \ref{Th_lower_bound}.}

\end{section}

\begin{section}{Study of self-similar solutions for the DNLE}\label{Sim_Sol_Subsection}

In this section we will illustrate a short description of the self-similar solutions of the DNLE, focusing on the properties useful for the proofs in the
paper. A complete analysis of these solutions is beyond the purpose of our paper and it can be the subject of a future work.

For a complete characterization of self similar solutions of the PLE in the case $p>2$
we refer to \cite{CambridgeJournals:4773264}. Likewise, for the relation between self-similar solutions of the PLE and those of the
PME we make reference to \cite{MR2378087}.

\emph{Self-similar solutions} of the DNLE are functions of the form
$$\mathcal{U}(t,x)=(t+s)^{-\alpha}h(r),\quad r=|x|(t+s)^{-\beta},$$
where $s\geq0$ is a constant, $\alpha$ and $\beta$ are positive parameters related by
\begin{equation}\label{relation_alpha_beta}
(m(p-1)-1)\alpha + p \beta =1.
\end{equation}
The profile $g:=h^m:[0,\infty)\rightarrow \mathbb{R}$ is a function satisfying the differential equation
\begin{equation}\label{odeDNLE}
 \alpha g^{\frac{1}{m}}(\eta) + \beta r \left(g^{\frac{1}{m}}\right)'(r) + \frac{N-1}{r}|g'(r)|^{p-2}g'(r)+ (p-1)|g'(r)|^{p-2}g''(r)=0,  \quad r>0,
\end{equation}
also written in the equivalent form
\begin{equation}\label{odeDNLE_equiv_form}
 \alpha h(r) + \beta r h'(r) + \frac{1}{r^{N-1}}\left(r^{N-1}|g'(r)|^{p-2}g'(r))\right)'=0,  \quad r>0.
\end{equation}

Self-similar solutions are (possibly signed) solutions of the DNLE in the whole space. When the support of the positive part of
such a function $\mathcal{U}$ is included in $\Omega$ then $\mathcal{U}_+$ is a sub-solution of the DNLE equation. For this reason, self-similar solutions are
useful to indicate the behaviour of a general solution $u$ of the DNLE.

We consider the initial conditions
\begin{equation}\label{init_cond}
 h(0)=M^m , \quad h'(0)=0.
 \end{equation}

The existence of a positive solution of ODE \eqref{odeDNLE} with initial conditions \eqref{init_cond}  on an interval $[0, a$), where $a \in (0,\infty]$,  can be proved using fixed point methods when $p\leq 2$ and  using phase plane methods when $p>2$  (we refer to \cite{CambridgeJournals:4773264} when $p>2$ where the author discusses the case of the PLE). As far as we know, fixed point methods do not work when  $p>2$.

If one multiplies the ODE by $r^{N-1}$ and integrate between $0$ and $r$, where $r \in [0,a)$, it follows that
\begin{equation}\label{formula_g_deriv}
|g'(r)|^{p-2}g' (r)= -\beta r g^{\frac{1}{m}}(r) - \frac{\alpha-\beta N}{r^{N-1}}\int_0^{r}s^{N-1}g^{\frac{1}{m}}(s) ds,
\end{equation}
or, equivalently,
\begin{equation}\label{formula_g_deriv2}
|g'(r)|^{p-2}g' (r)+ \frac{\beta N-\alpha}{Nr^{N-1}}\int_0^{r}s^{N}(g^{\frac{1}{m}})'(s) ds= -\frac{\alpha}{N} g^{\frac{1}{m}}(r) .
\end{equation}

We will make a formal study of self-similar solutions of the DNLE by considering the following cases: $\beta=0$, $\alpha-\beta N=0$,  $\alpha-\beta N>0$
and $\alpha-\beta N<0$. We define the numbers
\begin{equation*}
\alpha_B:=\frac{1}{m(p-1)-1+(p/N)}, \quad \beta_B:=\frac{\alpha_B}{N}=\frac{1}{(m(p-1)-1)N+p}, \quad \alpha_0:=\frac{1}{m(p-1)-1}.
\end{equation*}

\begin{figure}[!h]\label{figure_self_sim_sol}
	    \centering
        \subfigure[Case $\beta=0, \alpha=\alpha_0$]{
	    \includegraphics[width=70mm,height=46mm]{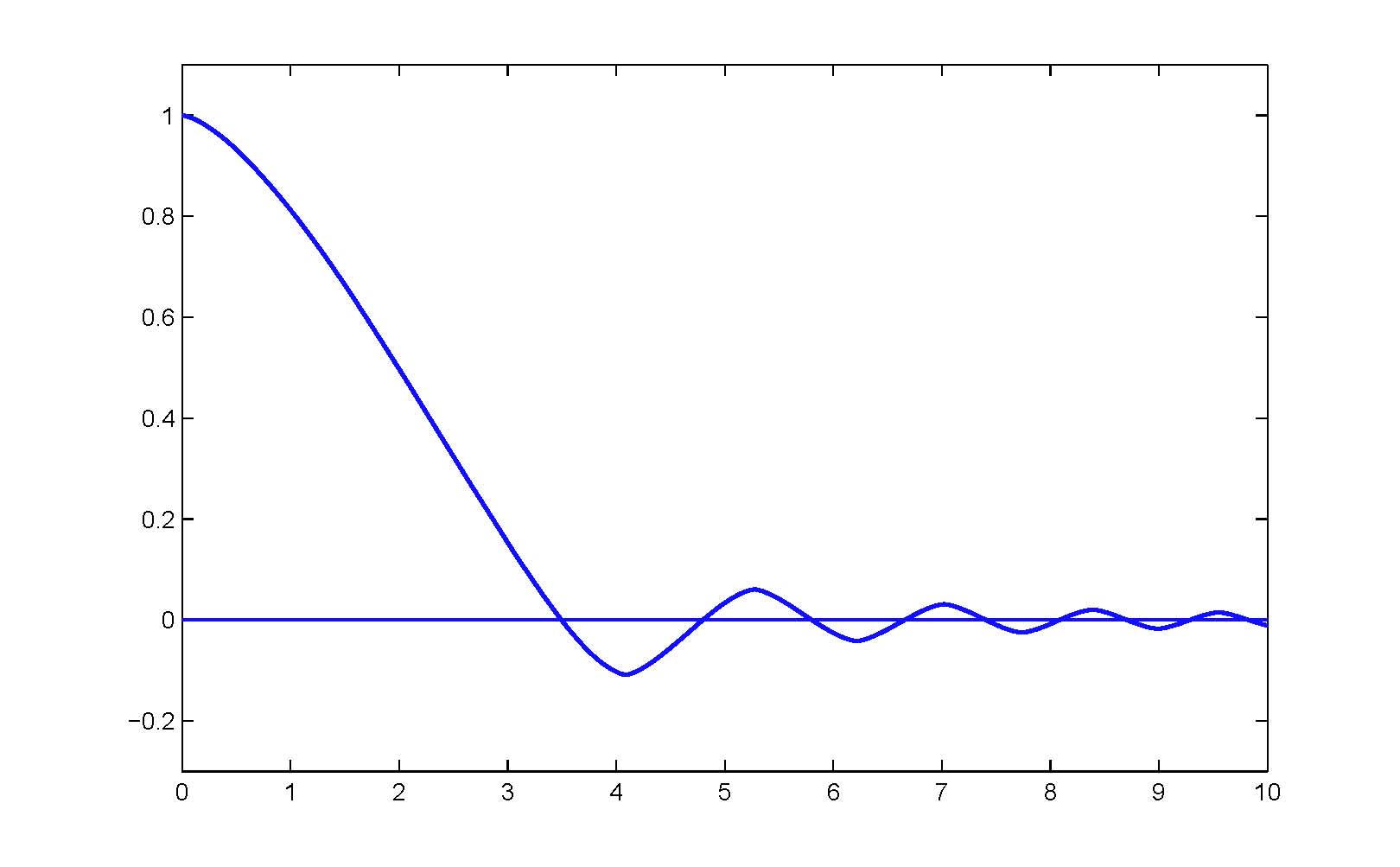}
	    }\qquad
        \subfigure[Case $\alpha=\beta N$]{
	    \includegraphics[width=60mm,height=46mm]{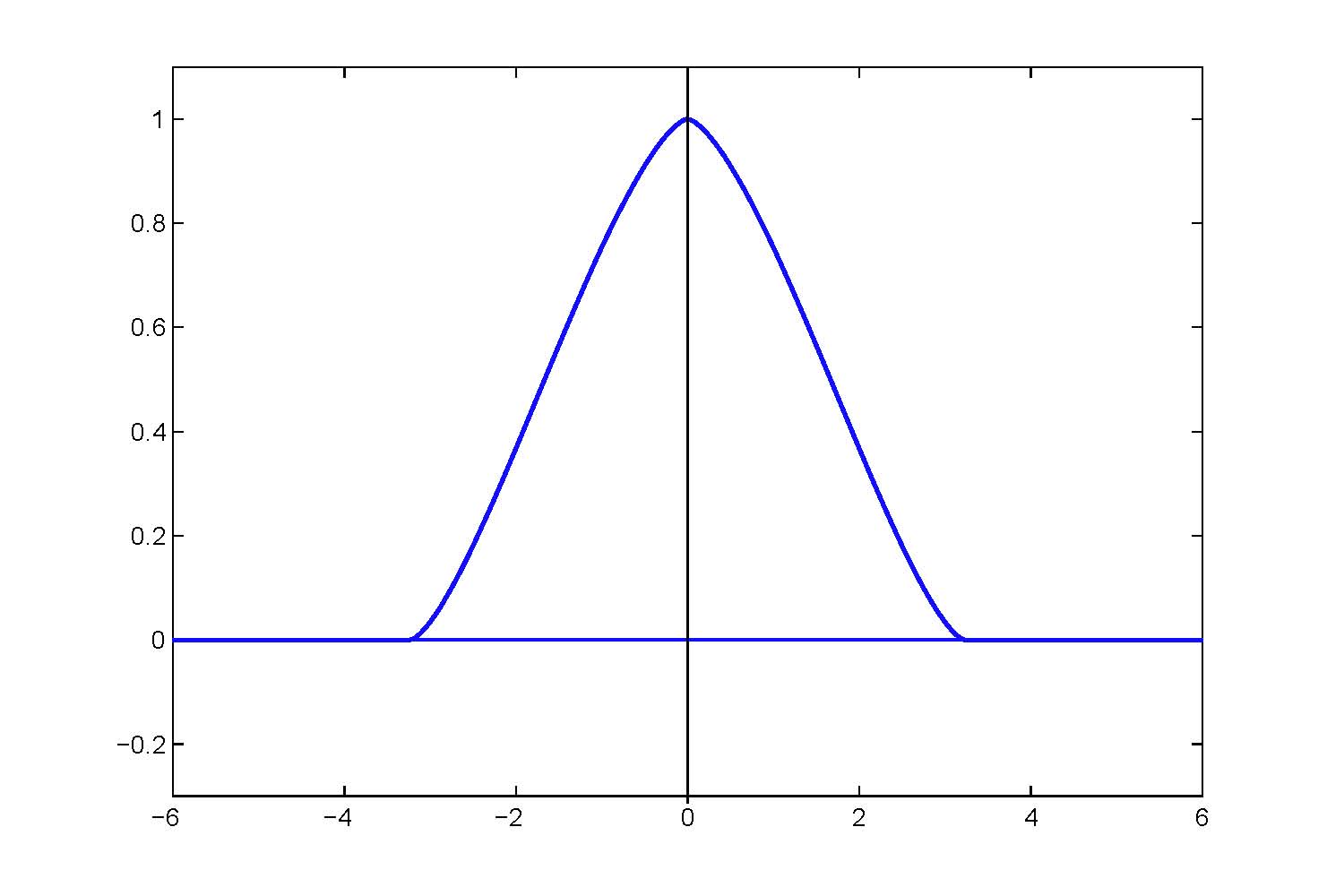}
	    }
	    \caption{Self-similar solutions of the DNLE}
\end{figure}

\noindent \textbf{I. Case $\beta=0, \ \alpha=\alpha_0$. Separate variables solutions}

\noindent They have the form $$U(t,x)=(t+s)^{\mu}f(x),$$
where $s>0$ is a constant and $f$ is the solution of the stationary problem \eqref{eq_f_DNLE}.
Notice that the functions belonging to this family are self-similar solutions according to the previous definitions when $f$ is a radial function. These functions are very useful since they indicate the asymptotic behaviour of $u$, the general solution of the DNLE and thus we will use them for comparison.
A second aspect is that they do not propagate and thus we have to consider also different types of self-similar solutions.

\medskip
\noindent \textbf{II. Case $\alpha=\beta N$. Barenblatt solutions}

 For more details we refer to \cite{JLVSmoothing}. A Barenblatt solution (also called source type solution) exists for the DNLE in the "good range" $$m(p-1)+\frac{p}{N}>1$$ that includes of course $m\geq 1$ (the degenerate PME) and $p\geq 2$  (the degenerate PLE).
When, moreover, $m(p-1)>1$ (our case), Barenblatt solutions have the form:
$$\mathcal{U}(x,t;a,s)=c(t+s)^{-\alpha}\left(a^{\frac{p}{p-1}}-|x(t+s)^{-\beta}|^{\frac{p}{p-1}}\right)_+^{\frac{p-1}{m(p-1)-1}},$$
where $s>0$ is a positive parameter and
\begin{equation}\label{alpha_beta_c}
\alpha=\alpha_B=\frac{1}{m(p-1)-1+(p/N)}, \quad \beta=\beta_B=\frac{\alpha_B}{N}, \quad c=\left(\frac{m(p-1)-1}{p} \left (\frac{\alpha}{N}\right)^\frac{1}{p-1} \right)^{\frac{p-1}{m(p-1)-1}}.
\end{equation}
When $s=0$, this function has a Dirac delta as initial trace
$$\lim_{t \rightarrow 0}\ \mathcal{U}(x,t)=M \delta_0(x).$$
The remaining parameter $a>0$ is free and can be uniquely determined in terms of the initial mass $\displaystyle{\int \mathcal{U} dx=M_0}.$

Barenblatt solutions are compactly supported and they propagate with finite speed. We will use them as a lower bound in order to prove the positivity
of $u$ inside $\Omega$. Since they have a flat landing contact (zero derivative at the boundary of their support) we can not obtain a
quantitative lower bound for $u$ up to the boundary of $\Omega$. Their advantage is that they have an explicit formula which is very advantageous for computations.

\medskip
\noindent \textbf{III. Case $\alpha> \beta N$. Intermediate self-similar solutions}

In this case $$\alpha>\alpha_B, \quad 0<\beta<\beta_B.$$
This family of self-similar solutions, which we denote by $\mathcal{V}$, inherits some useful properties of the Barenblatt solutions and the separate variables solutions: $\mathcal{V}$ has a compact support that propagates and $g=h^m$ has a transversal cross through the $r$ axis.
This is explained as follows. Consider $[0,a)$ the largest interval starting from $0$ where $h>0$. Then, from \eqref{formula_g_deriv} we obtain that
$g' (r)<0$ for  $r \in [0,a)$ and thus
$$|g'(r)|^{p-2}g' (r)\leq -\beta r g^{\frac{1}{m}}(r).$$
Furthermore, this implies that
$$-g'(r) \geq \beta^{\frac{1}{p-1}} r^{\frac{1}{p-1}} g^{\frac{1}{m(p-1)}}.$$
Integrating from $0$ to $r$ with $g(0)=M^m$ we obtain that
$$g(r) \leq \left(M^{\frac{m(p-1)-1}{p-1}}- \frac{m(p-1)-1}{mp}\beta^{\frac{1}{p-1}} r^{\frac{p}{p-1}}\right)^{\frac{m(p-1)}{m(p-1)-1}},$$
for all $r \in (0,a).$
From this upper bound we derive an estimate for the point $a$ where $h(a)=0$
$$a\leq  M^{\frac{m(p-1)-1}{p}} \left( \frac{mp}{m(p-1)-1} \right)^{\frac{p-1}{p}} \beta^{-\frac{1}{p}}.$$
The important fact is that $a$ is finite, thus $\mathcal{V}$ has a transversal cross through the $r$ axis at the point $r=a$ with
\begin{equation}\label{prop_g}
g' (a) = - \left(\frac{\alpha-\beta N}{a^{N-1}}\int_0^{a}s^{N-1}g^{1/m}(s) ds.\right)^{\frac{1}{p-1}}=:k_0 < 0.
\end{equation}
We will use the form
$$\mathcal{V}(x,t;M,s)= (t+s)^{-\alpha}[h(r;M)]_+,$$
and therefore the following characterization of the support
$$\text{supp }\mathcal{V}(x,t;M,s)=  \{(x,t):|x|\leq a(t+s)^{\beta}, t\geq 0 \}.$$
We denote by
$$
a=a(M) , \quad k=k(M),\quad h(\cdot)=h(\cdot;M)
$$
in order to emphasize their correspondence to the Cauchy problem with initial conditions $ h(0)=M, \  h'(0)=0.$
We remark that
\begin{equation}\label{constants_4}
h(a(M);M)=0, \quad h(r;M)=M h(M^{-\frac{m(p-1)-1}{p}}r;1) \quad \text{  and  } a(M)=M^{\frac{m(p-1)-1}{p}}a(1).
\end{equation}
Subsequently we will consider the self-similar solutions described above only on $[0,a)$, the largest interval starting from $0$ where they are positive;
 for complete definition, on $[a, \infty)$ they are assigned zero values. This way, they are sub-solutions of the DNLE.

\noindent \textbf{Remark.} In the present work we do not study the behavior of these functions when then $g$ takes negative values.
Depending on the values of $\alpha$ and $\beta$, $g$ can behave differently, as we can see in Figure 3.

\begin{figure}[!h]\label{figure2_self_sim_sol}
	    \centering
\subfigure[Case $\alpha>\beta N$]{
	   \includegraphics[width=50mm,height=30mm]{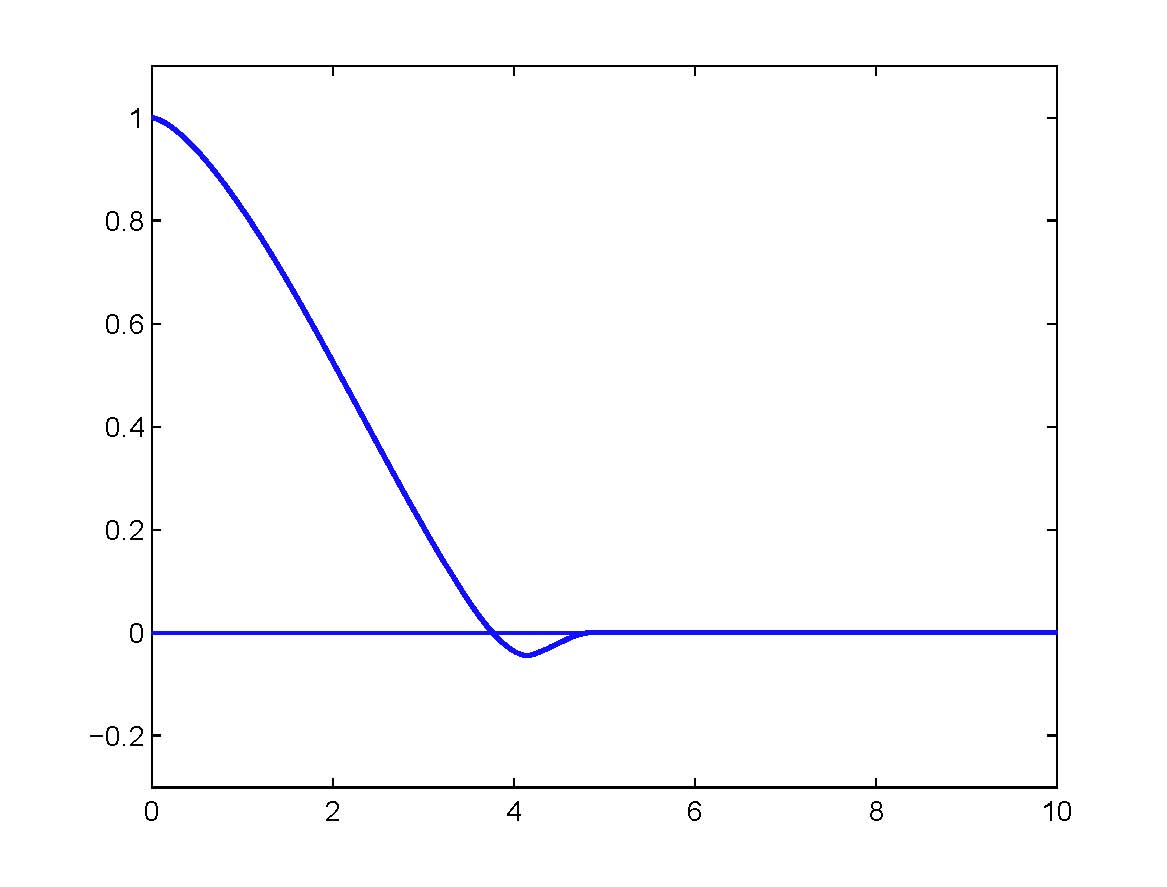}}
        \subfigure[Case $\alpha>\beta N$]{
	   \includegraphics[width=50mm,height=30mm]{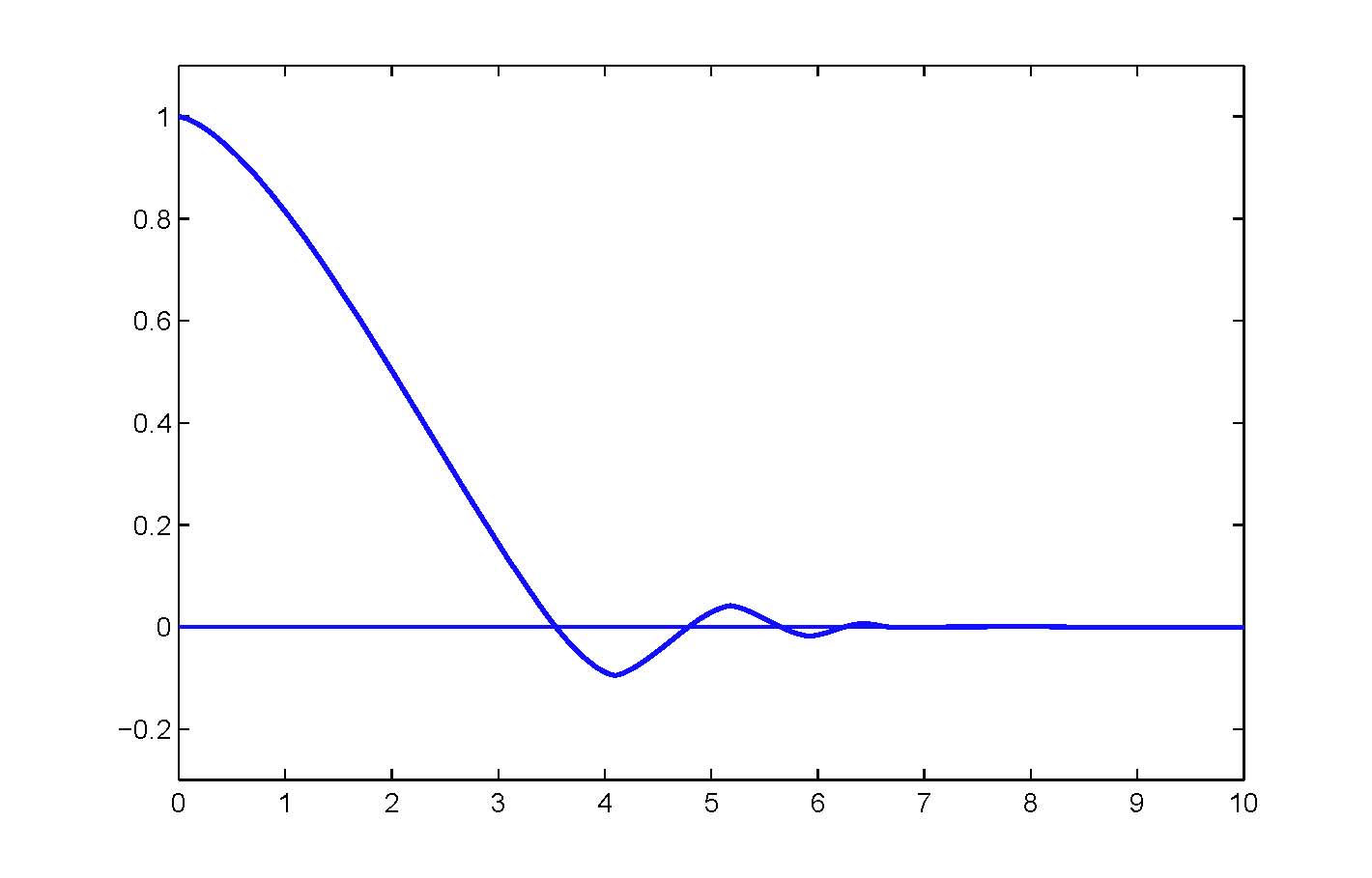}}
        \subfigure[Case $\alpha<\beta N$]{
	    \includegraphics[width=50mm,height=30mm]{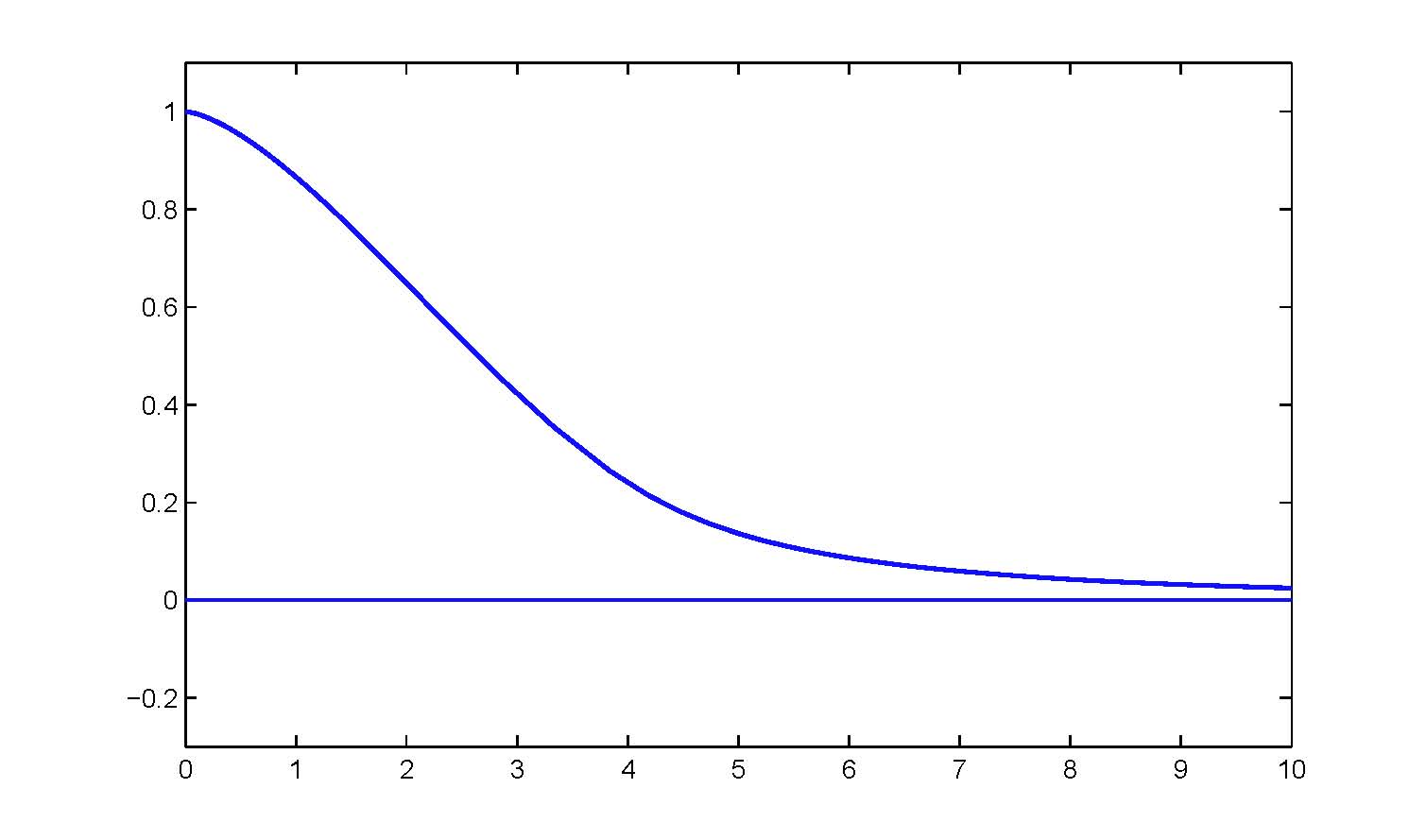}
	    }
	    \caption{Self-similar solutions of the DNLE}
\end{figure}

\newpage

\noindent \textbf{IV. Case $\alpha< \beta N$}

The self-similar function corresponding to the profile $g$ in this case does not have a compact support, hence this class is not useful for our estimates.
For completeness we will provide a formal characterization of these functions.

Recall that in this case $$0<\alpha<\alpha_B, \quad \beta>\beta_B.$$

Using basic computations as in the previous case, one can easily prove that $g$ is a positive decreasing function converging to $0$ when $r$ goes to $\infty$.
Moreover, for every $r_0>0$ there exists $C=C(r_0)>0$ such that
 \begin{equation}\label{lower_estim_g}
 g(r)\geq C r^{-\frac{\alpha m}{\beta}}, \quad \forall r\geq r_0.
 \end{equation}

\textbf{Asymptotic decay.} We point out that $g$ behaves as $r \rightarrow \infty$ like $G_a(r)=a r^{-\gamma}$, $\gamma=\alpha m/ \beta.$
We continue with a formal proof.

\noindent We will consider the following series expansions of $g$:
$$g(r)= a_1 r^{-\gamma}+ ... $$
as $r$ goes to $\infty$, where $\gamma>0$ is an exponent to be determined and "..." is representative for lower order terms.
Then
\begin{align*}
\alpha h(r) + \beta r h'(r) &= a_1^{\frac{1}{m}}\left(\alpha - \beta \frac{\gamma}{m}\right)r^{-\frac{\gamma}{m}}+...,\\
\frac{1}{r^{N-1}}\left(r^{N-1}|g'(r)|^{p-2}g'(r))\right)' &=a_1^{p-1}\gamma^{p-1}( \gamma(p-1)+p-N)r^{-(\gamma+1)(p-1)-1}+....
\end{align*}

For a comparison of the first terms we notice that
$$-(\gamma+1)(p-1)-1 < -\frac{\gamma}{m}.$$
\noindent Thus \emph{the leading asymptotic term} in the expansion of the ODE formula \eqref{odeDNLE_equiv_form} is
$$a_1^{\frac{1}{m}}\left(\alpha - \beta \frac{\gamma}{m}\right)r^{-\frac{\gamma}{m}}.$$
Moreover, we notice that the coefficient $\alpha - \beta \gamma/m=0,$ from where we deduce that \emph{exponent of the leading asymptotic term} is
\begin{equation}\label{gamma_formula}
\gamma=\frac{\alpha m}{\beta}.
\end{equation}
At this time we have no information about $a_1$. The coefficient of the remaining term is
$$
a_1^{p-1}\gamma^{p-1}( \gamma(p-1)+p-N),$$
 whose sign depends on the values of $\beta.$

Let $$\beta_B:=\frac{1}{\left( m(p-1)-1+p/N \right)N}, \quad \alpha_B:=\frac{1}{m(p-1)-1+p/N },$$
$$\gamma_1:= \frac{N-p}{p-1}, \quad \beta_1:=\frac{m(p-1)}{\left( m(p-1)-1+p/N \right)N}=m(p-1)\beta_B, \quad \alpha_1:=\frac{1-p\beta_1}{m(p-1)-1}.$$

The sign of coefficient $\gamma(p-1)+p-N$ can be now obtained depending on the values of $\beta.$
\begin{enumerate}
  \item Case
  $\gamma(p-1)+p-N>0 \Leftrightarrow \left\{
    \begin{array}{ll}
      \gamma>\gamma_1, \ \beta<\beta_1, \ \alpha>\alpha_1; \\[2mm]
      \text{or }p \geq N.
    \end{array}
  \right.
  $

  \item Case
$
\gamma(p-1)+p-N<0 \Leftrightarrow \left\{
    \begin{array}{ll}
      \displaystyle{\beta\in (\beta_1, 1/p), \ \alpha<\alpha_1,} & \hbox{$p<N$;} \\[2mm]
      \text{impossible}, & \hbox{$p \geq N$.}
    \end{array}
  \right.
  $

  \item Case $\gamma(p-1)+p-N=0 \ \Leftrightarrow \ \gamma=\gamma_1,\ \beta=\beta_1, \ \alpha=\alpha_1$.
  Then $g(r)=r^{-\gamma}$ is the solution of the ODE \eqref{odeDNLE}.
\end{enumerate}

We can observe that in the first two cases we can not deduce the decay of $g$ and we need to perform a second approximation.
\end{section}

\begin{section}{The quasilinear case $m(p-1)=1$}\label{Section_Quasilinear}

In this section we consider Problem \eqref{DNLE} for $m>0$, $p>1$, posed in a bounded domain $\Omega \in \mathbb{R}^N$ with smooth boundary of
class $C^{2,\alpha}$ and initial data $u_0 \geq 0$, $u_0 \in L^1(\Omega)$.  We study the large-time asymptotic behaviour of solutions of the problem
in the quasilinear case $m(p-1)=1$.

As usual, the problem is better understood via the method of rescaling. We consider
\begin{equation}\label{formula_v_quasilineal}
v(t,x)=e^{\lambda t}u(t,x), \quad t\in [0,\infty), x\in \Omega,
\end{equation}
where $\lambda$ is a real parameter whose choice we will justify in the next subsection.
Then $v$ is a solution of the \emph{rescaled problem}
\begin{equation}\label{eq_v_quasilinear}
  \left\{ \begin{array}{ll}
  v_{t}(t,x) = \Delta_p v^m(t,x)+ \lambda v(t,x) &\text{for }t>0 \text{ and } x \in \Omega, \\
  v(0,x)  =u_0(x) &\text{for } x \in \Omega, \\
  v(t,x)=0   &\text{for }t>0 \text{ and } x \in \partial \Omega.
    \end{array}
    \right.
\end{equation}

 The study of the asymptotic behaviour in the present case $m(p-1)=1$ differs considerably from the case $m(p-1)>1$ previously studied for several reasons.
 Firstly, the proof in the case $m(p-1)>1$ is based on monotonicity: the rescaled solution $t^{{\mu}}u(t,x) \nearrow v(t,x)$.
 This argument cannot be applied in the present case. Besides, we have no universal a-priori estimates similar to the degenerate case.

\subsection{The associated stationary problem}\label{subs5.1}

Consider the stationary problem associated to Problem \ref{eq_v_quasilinear}:
 \begin{equation}\label{eqf}
  \Delta_p f^m + \lambda f =0  \quad  \text{in } \Omega,  \qquad f(x)=0   \quad \text{for } x \in \partial \Omega,
    \end{equation}
that can be rewritten in terms of $z=f^m$ as
\begin{equation}\label{eqZ}
\Delta_p z+ \lambda |z|^{p-2}z=0   \quad \text{in } \Omega, \qquad  z=0   \quad \text{on }\partial \Omega.
    \end{equation}
We want to obtain solutions $z(x)>0$ in $\Omega$. We call \emph{eigenvalues} the $\lambda$-s for which there exists a nontrivial solution of Problem \ref{eqZ}, which is known as the eigenvalue problem for the $p-$Laplacian. The following result was proved in \cite{Anane,Lindqvist}.

\medskip

\begin{theorem}\label{Theorem_simplicite} (\textbf{Simplicity and isolation of the first eigenvalue of Problem {\rm \ref{eqZ}}})
The first eigenvalue $\lambda_1$ of Problem {\rm \ref{eqZ}} is simple and isolated. Moreover, $\lambda_1$ is the unique positive eigenvalue of  Problem {\rm \ref{eqZ}} having a nonnegative eigenfunction.
\end{theorem}
\noindent Moreover, the first eigenvalue  $\lambda_1$ of Problem {\rm \ref{eqZ}} can be characterized as
\begin{equation}\label{Lambda}
\lambda_1:=  \inf \left\{ \int_{\Omega}\frac{|\nabla \varphi |^p}{| \varphi |^p}  , \ \varphi \in W^{1,p }_0 (\Omega) \right\},
\end{equation}
that is $\lambda_1=\mathcal{C}^{-p}$ where $\mathcal{C}$ is the best constant of the embedding $W^{1,p }_0 (\Omega)$ into $L^p(\Omega)$.

\subsection{Preliminary estimates for the evolution problem}\label{subsection_preliminary_quasilineal}

We state two results obtained by Manfredi and Vespri (Theorems $1.4$ and $1.4$ from \cite{ManfrediVespri}).

\medskip

\begin{theorem}\label{th14ManfrediVespri}
Consider $m(p-1)=1$. Let $\Omega \subset \mathbb{R}^N$ be an open, bounded domain. Then there exists a unique solution of the problem \ref{DNLE} corresponding to an initial datum $u_0\in L^1(\Omega).$  Moreover, for all $t\geq 1$, there exists a constant $c(t)$ such that
\begin{equation}\label{estim1}
|u(t,x)|\leq \gamma_1 e^{-\lambda_1 t} c(t) f(x), \quad \forall x \in \Omega,
\end{equation}
where $f$ is a solution of the problem \ref{eqf} such that $\displaystyle{f\in C^0(\overline{\Omega}), \ f^{\frac{p-2}{p-1}}\nabla f \in L^p(\Omega)}$ and
 $\gamma_1$ is a positive constant depending only on the data $N$, $p$, $m$, the $L^1$ norm of $u_0$ and the $C^{1,\alpha}$ norm of $\partial \Omega.$
\end{theorem}

\medskip

\begin{theorem}\label{th15ManfrediVespri}
Consider the hypothesis of the previous theorem and assume moreover that $u_0 \geq 0$ and not identically zero. Then, for every $t\geq 1$, there exist the  constants $ \underline{c}(t), \ \overline{c}(t) \in \mathbb{R_+}$ such that  the following estimate holds
\begin{equation}\label{estim2}
e^{-\lambda_1 t} \underline{c}(t) f(x)\leq u(t,x)\leq e^{-\lambda_1 t} \overline{c}(t) f(x), \quad \forall x \in \Omega.
\end{equation}
Moreover we also have
\begin{equation}\label{estim3}
\gamma_1(t) e^{-\lambda_1 t} d(x)^{p-1} \leq u(t,x) \leq \gamma_2(t) e^{-\lambda_1 t} d(x)^{p-1},\quad \forall x \in \Omega, t>1
\end{equation}
and
\begin{equation}\label{estim4}
|\nabla u(t,x)| \leq \gamma_3(t) e^{-\lambda_1 t} d(x)^{p-2}, \quad \forall x \in \Omega, t>1.
\end{equation}
\end{theorem}
\noindent \textbf{Remarks}
\begin{itemize}
  \item \textbf{Reduction.}
  The estimates given by the previous theorems are true for every $t\geq t_0$, where $t_0>0$ is fixed. Since the doubly nonlinear equation is invariant under a time displacement, we can assume that the previous estimates are valid for every $t\geq 0$, otherwise we can start with initial data $u(t_0)$. We assume therefore such a displacement in time has been done.
  \item We can fix $f$ in any way, up to a multiplicative constant. Therefore we fixe $f$ a nonnegative solution of the problem \ref{eqf}.
\end{itemize}

\medskip

 Inspired by the ideas from \cite{Kamin_Peletier_Vazquez_1991} we will obtain more information about the constants $\overline{c}(t)$ and $\underline{c}(t)$. Let us define
\begin{equation}\label{c(t)}
\overline{c}(t)=\inf\{c: v(t,x) \leq c f(x)\}, \quad
\underline{c}(t)=\sup\{c: v(t,x) \geq c f(x)\}.
\end{equation}
According to the Theorem \ref{th15ManfrediVespri} the previous definition makes sense: $\overline{c}(t)<\infty$ and $ \underline{c}(t)>0$ for every $t\geq 0$.
Thus, we can take $\underline{c}(t)$ and $\overline{c}(t)$ to be the best constant such that estimate \eqref{estim2} holds.
It is a simple consequence of the Maximum principle that $\overline{c}(t)$ and $\underline{c}(t)$ are two decreasing, respectively increasing functions of $t$.
Therefore the following limits exist:
\begin{align}\label{limit_c}
\overline{c}_{\infty}&= \lim_{t \rightarrow \infty} \overline{c}(t), \quad \overline{c}(t) \searrow \overline{c}_{\infty}, \\
\underline{c}_{\infty}&= \lim_{t \rightarrow \infty} \underline{c}(t), \quad \underline{c}(t) \nearrow \underline{c}_{\infty}.
\end{align}

In addition, the constants  $ \underline{c}(t)$ and $ \overline{c}(t)$ are uniformly bounded
\begin{equation}\label{ineq_constants}
C_0\leq \underline{c}(t) \leq \underline{c}_{\infty}\leq \overline{c}_{\infty}\leq  \overline{c}(t) \leq C_1, \quad \forall t \geq 0.
\end{equation}

\medskip

\noindent We can sum up what we have proved so far in the following lemma.

\begin{lemma}\label{Lemma_bound_v}
Let $v$ be a solution of the rescaled problem \eqref{formula_v_quasilineal}. Then there exist the positive constants $C_0, \ C_1, \ C_2>0$  such that
\begin{equation}\label{bound_v1}
\underline{c}_{\infty} f(x) \leq v(t,x) \leq \overline{c}_{\infty}f(x), \quad \forall x \in \Omega, t\geq 0,
\end{equation}
\begin{equation}\label{bound_v}
C_0 d(x) \leq v(t,x) \leq C_1 d(x), \quad \forall x \in \Omega, t\geq 0,
\end{equation}
and
\begin{equation}\label{bound_grad_v}
|\nabla v(t,x)| \leq C_2 d(x)^{p-2}, \quad \forall x \in \Omega, t\geq 0,
\end{equation}
where $C_0,C_1,C_2>0$ depend on $\Omega$, $\lambda_1$ and $f$ is the positive solution of problem \ref{eqf} we have taken.
\end{lemma}

As a consequence we obtain the uniform convergence, up to subsequences, of $v(t,\cdot)$ to a stationary profile.

\medskip

\begin{theorem}\label{pro_unif_conv_quasilinealcase}(\textbf{Uniform convergence to an asymptotic profile up to sequences})\\
\noindent Consider $m(p-1)=1$. Let $\Omega\subseteq {\mathbb R}^N$ be a  bounded domain of class $C^{2,\alpha}$, $\alpha>0$. Let $u(t,\cdot)$ be a weak
solution to Problem \ref{DNLE} corresponding to the nonnegative initial datum $u_0\in L^1(\Omega)$. Then for any given $T>0$ there exists a sequence
$\tau_n \rightarrow \infty$ such that
\begin{equation}\label{conv_v_f_quasilineal}
|e^{\lambda_1 (\tau_n+t)} u(\tau_n+t,s) - c_*f(x) | \rightarrow 0 , \quad \tau_n \rightarrow \infty,
\end{equation}
uniformly for $x \in \Omega$ and $0\leq t\leq T$, where $f$ is the positive solution of problem \ref{eqf} we have taken and $c_*$ is a positive constant.
\end{theorem}

\begin{proof}

\noindent  \textbf{I. Energy estimates}

\noindent \textbf{Ia.} We consider the following energy functional
$$E(t)=E[v(t)]:=\frac{1}{p} \int_{\Omega} | \nabla v^m (t,x)|^p dx  -  \frac{m}{m+1}\lambda_1 \int_{\Omega}v^{m+1}(t,x) dx.$$
We compute the energy dissipation
$$-\frac{d}{dt}E(t)=I(t)=m \int_{\Omega}v^{m-1}  v_t^2 dx \geq 0 ,
$$
which means that $E(t)$ is a non-increasing function and
$$
E(t_1)-E(t_2)=\int_{t_1}^{t_2}I(t)dt.
$$
As well, we deduce that the integral
$$\int_{t_1}^{t}I(t)dt$$
is convergent as $t\rightarrow \infty$ and $E(t)$ has a limit as $t\rightarrow \infty$.

Since $v(t,x)$ is bounded in $\Omega$ uniformly for $t\geq 0$ we obtain that
\begin{equation}\label{estim_Lp_norm_quasilineal}
\int_{\Omega} | \nabla v^m (t,x)|^p dx \leq M , \quad \forall t\geq 0,
\end{equation}
in other words $ |\nabla v^m(t,\cdot)|$ is uniformly bounded in $L^p(\Omega)$ for $t\geq 0$.

\medskip

\textbf{Ib}. As a consequence of \eqref{estim_Lp_norm_quasilineal} one can prove via H\"{o}lder's Inequality the following technical result
\begin{equation}\label{estim_int_t1t2}
\int_{\Omega}(\Delta_p v^m(t_1,x) )v^m(t_2,x)dx \leq M, \quad \forall t_1,t_2 \geq 0.
\end{equation}

\medskip

\noindent \textbf{II. Convergence.} We define
\begin{equation}\label{vtilde}
\tilde{v}_{\tau}(t,x)=v(t+\tau,x), \quad t,\tau>0, \ x \in \Omega.
\end{equation}
Then $\tilde{v}_{\tau}$ is still a solution of Problem \ref{eq_v_quasilinear} with initial data $\tilde{v}_{\tau}(0,x)=v(\tau,x).$

We fix $T>0$. The family $(\tilde{v}_{\tau})$ is relatively compact in
$$
X=L^{\infty}([0,T]\times \overline{\Omega})
$$
thus it converges along subsequences
$$
v_{{\tau}_n}(t,x)\to S(t,x) \qquad \mbox{uniformly in } \ (t,x)\in [0,T]\times \Omega.
$$
From the a-priori estimates we deduce the boundedness of $v$
$$\tilde{C}_0 \leq v(t,x) \leq \tilde{C}_1, \quad \forall t\geq 0, \ x \in \Omega,$$
and since $S$ is the limit of $v_{\tau_n}(t,x)$, then $S$ also satisfies the same lower and upper bounds.

In what follows we fix such a subsequence $(\tau_n)$ and the corresponding limit $S(t,x)$.
\medskip

\noindent  \textbf{III. Convergence in measure of gradients}

Similar to the case $m(p-1)>1$ one can prove the convergence in measure of the sequence $(\nabla v^m_{\tau_n}(\cdot,\cdot) )_{n}$, where in the present case
$$v_{\tau_n}:[0,T]\times \Omega \rightarrow [0,\infty).$$
More exactly, we can prove by similar methods that the sequence $\displaystyle{(\nabla v^m_{\tau_n})_{n>0}}$ is Cauchy in measure, thus it converges in measure to a
function $W:[0,T]\times \Omega \rightarrow \mathbb{R}^N$.
It is a well known fact (Lemma \ref{Lemma_conv_measure} in the Appendix) that if a sequence is uniformly bounded in $L^p$ and converges in measure,
then it converges strongly in any $L^q $, for any $1\leq q <p$.
It follows that
$$\nabla v^m_{\tau_n} \rightarrow W \quad \text{ strongly in }(L^q([0,T]\times \Omega))^N \text{ when }\tau_n \rightarrow \infty, \text{ for every } 1\leq q <p.$$
Thus, we get that, up to subsequences,
\begin{equation}\label{conv_ae_gradients_quasilineal}
\nabla v^m_{\tau_n}(\cdot, \cdot) \rightarrow W(\cdot,\cdot) \quad \text{a.e. in }[0,T]\times \Omega,
\end{equation}
and we conclude that
$$W(t,x)=\nabla S^m(t,x).$$

\medskip

\noindent \textbf{IV. The limit is a solution of the stationary problem}

Multiply equation \eqref{eq_v_quasilinear} by any test function $\phi(x) \in C_c^{\infty}(\Omega)$ and integrate in space, $x \in \Omega$, and time between $\tau_n$ and $\tau_n + T$. We get that
\begin{equation}\label{eq1_v_phi_quasi}
\int_{\Omega}(v(\tau_n+T)-v(\tau_n)) \phi dx =- \int_{\tau_n}^{\tau_n+T}\int_{\Omega} |\nabla v^m|^{p-2}\nabla v^m  \nabla \phi dx dt  + \lambda_1\int_{\tau_n}^{\tau_n+T}\int_{\Omega} v \phi dx dt .
\end{equation}

\noindent (i) The left hand side term of \eqref{eq1_v_phi_quasi} is uniformly bounded independently of $T$:
\begin{equation}\label{eq2_v_phi_quasi}\left|\int_{\Omega}v(\tau_n+T) \phi dx - \int_{\Omega}v(\tau_n) \phi dx\right| \leq 2 \tilde{C}_1 |\Omega|\|\phi\|_{L^{\infty}(\Omega)}.
\end{equation}
Furthermore, if $\phi$ is supported in a compact $K\subset \Omega$ where $0<c_1 \le v\le c_2$  and $0<s\leq T$ then
\begin{equation}
\begin{array}{c}
\displaystyle \left|\int_{\Omega} (v(\tau_n+s) -v(\tau_n))\, \phi dx\right| \leq
\int_{\tau_n}^{\tau_n+s}\int_{\Omega}|v_t(t)\, \phi| dxdt \leq
 \\[8pt]
\displaystyle  \leq CT^{1/2}\left(\int_{\tau_n}^{\tau_n+s}\int_{\Omega} ((v^{(m+1)/2})_t)^2dxdt\right)^{1/2}=
\displaystyle  CT^{1/2}\left(\int_{t_n}^{\tau_n+s}I(t) dt\right)^{1/2}.
\end{array}
\end{equation}
Since the double integral $\int_1^\infty I(t) $ is finite, the integral on the right hand side goes to zero as $\tau_n\to\infty $.  On the other hand,
\begin{align*}\int_{\Omega}v(\tau_n+s) \phi dx - \int_{\Omega}v(\tau_n) \phi dx&=\int_{\Omega}v_{\tau_n}(s,x) \phi dx - \int_{\Omega}v_{\tau_n}(0,x) \phi dx\\
&\rightarrow \int_{\Omega}S(s,x) \phi dx - \int_{\Omega}S(0,x) \phi dx, \quad \tau_n \rightarrow \infty.
\end{align*}
Thefore, we showed that
$$\int_{\Omega}(S(s,x)-S(0,x))\, \phi \,dx=0,
$$
for any $0< s\leq T$ and any test function $\phi$ with $\text{supp }\phi=K\subset \subset \Omega$. Since $K$ is arbitrary, the result holds for all $\phi  \in C_c^{\infty}(\Omega)$
which implies that $S$ is independent of time on $[0,T]$:
\begin{equation}\label{V_indep_time}
S(s)=S(0), \quad \forall s\in [0,T].
\end{equation}

\noindent (ii) For the right hand side, we continue as follows.
Let $\tau_n\rightarrow \infty$. Since $\nabla v^m_{\tau_n} \rightarrow W=\nabla S^m$ a.e. in $[0,T]\times \Omega$ then using Lemma \ref{brezis} of the Appendix we get that
\[
\int_{\tau_n}^{\tau_n+T}\int_{\Omega} |\nabla v^m|^{p-2}\nabla v^m  \nabla \phi dx dt =\int_{0}^{T}\int_{\Omega} |\nabla v^m_{\tau_n}|^{p-2}\nabla v^m_{\tau_n}  \nabla \phi dx dt \rightarrow  T \int_{\Omega} |\nabla S^m|^{p-2}\nabla S^m  \nabla \phi dx.  
\]

The last integral
\[\int_{\tau_n}^{\tau_n+T}\int_{\Omega} v \phi dx dt=\int_{0}^{T}\int_{\Omega} v_{\tau_n}(t,x) \phi(x) dx dt\rightarrow T \int_{\Omega} S(t,x) \phi(x) dx, \quad \tau_n \rightarrow \infty.  
\]

Therefore
\begin{equation}\label{equality_limit}
-\int_{\Omega} |\nabla S^m|^{p-2}\nabla S^m   \nabla \phi dx +  \lambda_1\int_{\Omega} S \phi dxdt=0, \quad \forall t\in [0,T]
\end{equation}
and thus $S$ is a weak solution of the stationary problem \ref{eqf}. According to the Theorem \ref{Theorem_simplicite}
 $$S(t,x)=S(0,x)=c_*f(x), \quad \forall t\in [0,T], \ x\in \Omega.$$
For simplicity, we denote $S(x):=c_*f(x).$

\end{proof}

\noindent \textbf{Remarks}

\noindent   $\bullet$ $V=S^m$, where $S:=c_*f$, is a positive solution of the eigenvalue problem for the $p-$Laplacian:
\begin{equation}\label{eqV}
\Delta_p V + \lambda_1 V^{p-1} =0 \text{ in }\Omega, \quad V=0 \text{ on }\partial \Omega.
\end{equation}
\noindent   $\bullet$  As a consequence of \eqref{limit_c} and \eqref{conv_v_f_quasilineal} the constants $\underline{c}_{\infty}$, $\overline{c}_{\infty}$ and $c_*$ satisfy
\begin{equation*}\label{ineq_constants_quasilineal}
\underline{c}_{\infty} \leq c_* \leq \overline{c}_{\infty}.
\end{equation*}

\noindent   $\bullet$ The previous proposition does not guarantee the uniqueness of a stationary limit $S$. Therefore the rescaled solution $v(t,x)$ may oscillate
between the bounds $\underline{c}_{\infty}\,f(x)$ and $\overline{c}_{\infty}\,f(x)$ by converging on subsequences to asymptotic profiles of the form $c_*f$ with $\underline{c}_{\infty} \le c_*\le \overline{c}_{\infty}$. This kind of behavior has to be considered in the case of some parabolic evolution equations, for example in the case of signed solutions of the porous medium equation $u_t=\Delta u^m$, $m>1$. The set of possible asymptotic profiles is obtained as the $\omega$-limit of the solution and it is contained in the set of classical solutions of the associated stationary (elliptic) problem (we refer to the survey \cite{JLVsurvey}).

Next, we will prove that an oscillating behaviour is not possible. More, exactly, we show that
$$\underline{c}_{\infty}= c_* = \overline{c}_{\infty},$$
which guarantees the existence of a unique asymptotic profile $S=c_*f$ and therefore the uniform convergence of the rescaled solution $v(t,x)$ to $S$ for all times, uniformly in $\overline{\Omega}$.

To this aim, we will study the behaviour of the quotient $v/S$ up to the boundary.

\subsection{The relative error function and its equation}

\textbf{Assumptions (A)}. In what follows we make the assumptions:  $\Omega\subseteq {\mathbb R}^N$ is a bounded domain of class $C^{2,\alpha}$, $\alpha>0$, $u(t,\cdot)$ denotes the  weak solution to Problem \ref{DNLE} in the case $m(p-1)=1$ corresponding to the nonnegative initial datum $u_0\in L^1(\Omega)$.
We fix $T>0$, the corresponding sequence $\tau_n \rightarrow \infty$ and the constant $c_* \in [\underline{c}_{\infty}, \overline{c}_{\infty}]$ obtained in Theorem \ref{pro_unif_conv_quasilinealcase} for which the rescaled solution $v(t,x)=e^{\lambda_1 t} u(t,x)$ converges
$$
 \|e^{\lambda_1 (\tau_n+t)} u(\tau_n+t,\cdot) - c_*f(\cdot) \|_{L^{\infty}(\Omega)} \rightarrow 0 , \quad \tau_n \rightarrow \infty
$$
uniformly for $t\in [0,T]$, where $f$ is the positive solution of the problem \ref{eqf} we have taken. We denote $S:=c_*f$ and we call it \emph{possible asymptotic profile}.

Starting from this partial convergence result, we will obtain a much stronger convergence as $t\to\infty$: we show the uniqueness of the asymptotic profile
and the convergence in relative error of $v(t,\cdot)$ to $S(\cdot)$ up to the boundary as a consequence of the next proposition and estimates \eqref{bound_v1}.

\medskip

\begin{pro}\label{pro_behavior_bdry_REF}(\textbf{Behaviour up to the boundary})
Under the assumptions (A) there exists a unique constant $c_*>0$ depending on $u_0$ and $\Omega$, and for given $\epsilon>0$ there exists $t(\epsilon)>0$  such that
\begin{equation}
-\epsilon < \frac{v^m(t,x)}{S^m(x)} -1  < \epsilon, \quad \forall x\in \Omega, \quad \forall t\geq t(\epsilon).
\end{equation}
Moreover, $S=c_*f$ is the solution of the problem \ref{eqf} announced in Theorem \ref{pro_unif_conv_quasilinealcase}.
\end{pro}

\medskip

Motivated by the techniques used by Bonforte, Grillo and V\'{a}zquez in \cite{BDGV09} we will use the so called relative error function and the method of barriers.

To this aim, we introduce the \textit{Relative Error Function}(REF)
\begin{equation}\label{REFdef}
\phi(t,x)=\frac{v^m(t,x)}{S^m(x)}-1, \quad v^m =S^m(\phi+1)=V(\phi+1), \quad \text{and} \quad V=S^m.
\end{equation}

\medskip

\noindent {\sc Notations.} We define
$$\Omega_{I,\delta}= \{x \in \overline{\Omega}:d(x)> \delta\}, \quad \Omega_{\delta}=\Omega \setminus \overline{\Omega_{I,\delta}}= \{x \in \overline{\Omega}:d(x)< \delta\}.$$
where in what follows $\delta>0$ is considered to be a small positive parameter (see Subsection \ref{subsection_dist_bdry} of the
Appendix for properties of the distance to the boundary function).

\noindent \textbf{Properties of the REF}

\noindent $\bullet$ \emph{The parabolic equation of the REF.} Using the equations satisfied by $v$ and $V$ and relation $m(p-1)=1$ we obtain that
\begin{equation}\label{REFeq}
(p-1)(1+\phi)^{p-2}\phi_t= V^{-(p-1)}\Delta_p ((\phi+1)V)+ \lambda_1(\phi +1)^{p-1}.
\end{equation}

\noindent  $\bullet$ $\phi$ is uniformly bounded in $(t,x)$ for $t>0$. This can be derived from the estimates \eqref{bound_v} on $v$ and $S$, which is a stationary solution:
\[
\left(\frac{C_0}{C_1}\right)^m-1=C_{2,m}\leq \phi \leq C_{3,m}=\left(\frac{C_1}{C_0}\right)^m-1.
\]

\noindent  $\bullet$ In any interior region $\Omega_{I,\delta} \subset \Omega$, the REF function $\phi$ satisfies
$$\displaystyle{1+\phi = \frac{v^m}{V}>0} \quad \text{ in }  \Omega_{I, \delta}  \quad \text{for any } t\geq 0.$$

\noindent $\bullet$ \emph{Regularity of solutions of the parabolic equation \eqref{REFeq}}.
Since $\phi$ is also bounded in the interior of $\Omega$, we conclude that the parabolic equation \eqref{REFeq} is neither degenerate nor singular in the interior of $\Omega$. Also the solution $\phi$ of such a parabolic equation is H\"{o}lder continuous in any inner region $\overline{\Omega}_{I,\delta} \subset \Omega$ since both $v$ and $S$ are H\"{o}lder continuous and positive in the interior of $\Omega$.

\noindent \textbf{Convergence of the REF in an interior region of $\Omega$.} Under the running assumptions, we know by Theorem  \ref{pro_unif_conv_quasilinealcase} that
\[
\sup_{\overline{\Omega}}|v(\tau_n+t)-S|\rightarrow 0, \quad \text{as } n \rightarrow \infty,
\]
uniformly for $t\in [0,T]$ for a fixed $T>0$ and a corresponding sequence $(\tau_n)_n$, but this is not sufficient to prove the convergence of the
quotient $v^m/S^m$ to $1$ in the whole $\Omega$, since at the boundary there is the problem caused by the fact that both $v$ and $S$ are $0$ and therefore
the parabolic equation \eqref{REFeq} may degenerate at the boundary. However such a problem is avoided in any interior region where both $v$ and $S$ are strictly positive.

We can sum up the results we proved so far in the following lemma.

\medskip

\begin{lemma}\label{lemma_inner_conv}(\textbf{Inner convergence})
Let $v$ the solution of the rescaled problem \eqref{eq_v_quasilinear} and $S$ the solution of stationary problem \eqref{eqf}
corresponding to a given $T>0$ and a sequence $\tau_n \rightarrow \infty$ as in Theorem \ref{pro_unif_conv_quasilinealcase}.
 Let $\phi$ be the associated relative error function defined by \eqref{REFdef}. Then
\[
\| \phi(\tau_n+t,\cdot) \|_{L^{\infty}(\Omega_{I,\delta})}=\sup_{\Omega_{I,\delta}}|\phi(\tau_n+t,\cdot)| \rightarrow 0 ,
\]
as $ \tau_n \rightarrow \infty$ uniformly in $x\in \Omega_{I,\delta}$ and $0\le t\le T$, for any given $\delta>0$.
\end{lemma}

\subsection{Construction of the upper barrier and consequences}

In order to finish the proof of Theorem \ref{Th_conv_Rel_error_quasilinear} we have to prove the uniform convergence of $\phi$ up to the boundary.
This will be realized using a barrier argument based on the ideas from \cite{BDGV09}.

Throughout  the paper we will use the notation $\xi_0$ for the critical value which implies good properties of the  distance to the boundary
function in $\Omega_{\xi_0}$ as we explain in Subsection \ref{subsection_dist_bdry} of the Appendix.

Let us first point out some connections between distance to the boundary function and the solutions of the eigenvalue problem \eqref{eqV}.

\medskip

\begin{lemma}\label{lemma_properties_V}(\textbf{Properties of the asymptotic profile $V=S^m$})
Let $V$ be a solution of the eigenvalue problem \eqref{eqV}. Then $V$ satisfies the following estimates:
\begin{enumerate}
\item  There exist $C_{0}>0$ and $C_1>0$ such that
$$C_0^m d(x) \leq V(x) \leq C_1^m d(x), \quad \forall x \in \Omega.$$
  \item For every $0<\xi_1 \le \xi_0$ there exists a constant $\beta_0>0$ such that
$$\nabla V (x) \cdot \nabla d(x) \geq \beta_0 >0, \quad \forall x \in \Omega_{\xi_1}.$$
  \item For every $0<\xi_1 \le \xi_0$ there exists a constant $K_1>0$ such that
  \begin{equation}\label{bounds_gradientV}
   0<K_1 \leq |\nabla V| \leq K_2 , \quad \forall x\in  \Omega_{\xi_1}.
\end{equation}

\end{enumerate}
\end{lemma}
\begin{proof}
Point $1.$ is a consequence of Lemma \ref{Lemma_bound_v}.

The proof of the point $2.$ is similar to the one given \cite{BDGV09} since the function $V$ involved has the same properties as its correspondent in the fast diffusion problem.

Point $3.$ is a consequence of the estimate from point $1.$ and estimates  \eqref{bound_grad_v} and \eqref{estim_grad_d(x)} since
$$|\nabla V|= |mS^{m-1}\nabla S| \leq K_2 d(x)^{p-2+(m-1)/m}=K_2 d(x)^{m(p-1)-1}=K_2, \quad \forall x \in \Omega.$$

\end{proof}

Next, we present the construction of the barrier that plays an important role in the estimate of REF $\phi$ close to the boundary.
We mention that our construction is different from the one of \cite{BDGV09} where the operator was the usual Laplacian $\Delta$.
In our case, the $p$-Laplacian operator contains also mixed derivatives of second order whose estimate is more technical.

\medskip

\begin{lemma}\label{Lemma_barrier}(\textbf{Upper barrier})
We can choose positive constants $A,B,C$ so that for every $t_0>0$ the function
\begin{equation}\label{barrier}
\Phi(t,x)=C-BV(x)-A(t-t_0),
\end{equation}
is a super-solution to equation \eqref{REFeq} on a parabolic region near the boundary
$$\Sigma_{\Phi}=\{(t,x)\in (t_0,\infty): \Phi(t,x)\geq -1 \},$$
and moreover $\Sigma_{\Phi}\subset (t_0,T_0) \times \Omega_{\xi_1},$ where $\xi_1 \leq \xi_0$.
\end{lemma}
\begin{proof}
We will prove that the function \eqref{barrier} is a supersolution for the equation \eqref{REFeq} on the parabolic region $\Sigma_{\Phi}$ if we can find constants $A$, $B$ and $C$ such that
\begin{equation}\label{REFsupersolution}
(p-1)(1+\Phi)^{p-2}\Phi_t \geq V^{-(p-1)}\Delta_p ((\Phi+1)V)+ \lambda_1(\Phi +1)^{p-1}.
\end{equation}
We will prove that a convenient choice for $A$, $B$ and $C$ will be of the form
\begin{equation}\label{condABCsupersolution}
\left(\lambda_1(C +1) + A(p-1)\right) \xi_1^{p-1} \leq  \omega  B,
\end{equation}
where \begin{equation*}\label{omega}
\omega=\min\{1, 2^{2-p} \} \cdot \frac{2(p-1)K_1^p}{C_1 }.
\end{equation*}
From the beginning we assume that $(t,x) \in \Sigma_{\Phi} \subset (t_0,T_0) \times \Omega_{\xi_1}$ where $T_0$ is such that
\[
0<T_0-t_0 \leq \frac{C-BV(x) }{A}.
\]
The left hand side term satisfies
$$
(p-1)(1+\Phi)^{p-2}\Phi_t = -(p-1)A(1+\Phi)^{p-2}.
$$
The right hand side term is of the form
$$
V^{-(p-1)}\Delta_p ((\Phi+1)V)+ \lambda_1(\Phi +1)^{p-1}=V^{-(p-1)}\Delta_p f(V)+ \lambda_1(\Phi +1)^{p-1}
$$
where $$
f(z)=(C+1-Bz-A(t-t_0))z.
$$
The term $\Delta_p f$ can be computed as
  $$
 \Delta_p f(V)=|f'(V)|^{p-2}f'(V) \Delta_p V + (p-1)|f'(V)|^{p-2}f''(V)|\nabla V|^p.
  $$

\noindent  \textbf{Properties of the function $f$}
\begin{enumerate}
 \item Function $f$ is a concave parabola with zero values at the points $z=0$ and $z=z_0$ where
$$z_0:=\frac{C+1-A(t-t_0)}{B}.$$
  \item The derivatives are
  $$f'(z)=C+1-2Bz-A(t-t_0), \quad f''(z)=-2B.$$
  \item When applied to $V$, the derivative $$f'(V)=\Phi+1-BV.$$  Moreover, sufficiently close to the boundary, $f'(V)$ is positive and bounded. By choosing
\begin{equation}\label{choice2_xi1}
\xi_1 = \min\left\{\xi_0, \ \frac{1}{C_1^m}\frac{z_0}{4}= \frac{C+1-A(t-t_0)}{4BC_1^m} \right\},
\end{equation} we obtain the following bound on  $\Omega_{\xi_1}$:
$$
0<V(x) \leq C_1^m d(x) \leq C_1^m \xi_1 \leq \frac{z_0}{4} $$
and then $$  0<\frac{k_1}{2} \leq f'(V(x))\leq k_1,$$
where
\begin{equation}\label{k_1}
k_1:=f'(0)=C+1 - A(t-t_0),\qquad  \frac{k_1}{2}= f'\left(\frac{z_0}{4}\right)=\frac{1}{2}( C+1 - A(t-t_0))>0.
\end{equation}
Since $p>1$, we obtain a lower bound for $f'(V)^{p-2}$ on $\Omega_{\xi_1}$ as follows
\begin{equation}\label{lower_bound_f'(V)}
f'(V)^{p-2} \geq  \beta k_1^{p-2}, \qquad \beta:=\min\{ 1, 2^{2-p} \} .
\end{equation}
\end{enumerate}

\noindent \textbf{Sufficient conditions for the parameters}

Since $V$ is a solution of the stationary problem \eqref{eqV} then $-\Delta_p V=\lambda_1 V^{p-1}$ and inequality \eqref{REFsupersolution} can be rewritten as
$$
-A(p-1)(1+\Phi)^{p-2} \geq - \lambda_1|f'(V)|^{p-2}f'(V)  - 2B(p-1)V^{-(p-1)}|f'(V)|^{p-2}|\nabla V|^p + \lambda_1(\Phi +1)^{p-1}.
$$
The idea is that $V^{-(p-1)}$ can have large values close to the boundary, thus it is sufficient to find $A$, $B$ and $C$ such that
\begin{equation}\label{ineq_barrier_2}
   \lambda_1(\Phi +1)^{p-1} + A(p-1)(1+\Phi)^{p-2}\leq 2B(p-1)V^{-(p-1)}|\nabla V|^p |f'(V)|^{p-2}.
\end{equation}

For $\xi_1$ as in \eqref{choice2_xi1} and the bounds \eqref{bounds_gradientV}, \eqref{lower_bound_f'(V)}, the right hand side term of \eqref{ineq_barrier_2} satisfies the lower bound
$$
2B(p-1)V^{-(p-1)}|\nabla V|^p  |f'(V)|^{p-2} \geq 2B(p-1)\beta K_1^p k_1^{p-2} \left(C_1 \xi_1^{p-1} \right)^{-1} =:II.
$$
For the left hand side term of \eqref{ineq_barrier_2} on $\Sigma_{\Phi}$ we obtain the upper bound
$$
\lambda_1(\Phi +1)^{p-1} + A(p-1)(1+\Phi)^{p-2} \leq \lambda_1(C +1-A(t-t_0))^{p-1} + A(p-1)(C+1-A(t-t_0))^{p-2}:=I.
$$
Thus it is sufficient to take $A$, $B$ and $C$ such that
\begin{equation*}
(C+1-A(t-t_0))^{p-2} \left(\lambda_1(C +1-A(t-t_0)) + A(p-1)\right)\leq 2B(p-1)\beta K_1^p k_1^{p-2} \left(C_1 \xi_1^{p-1} \right)^{-1} .
\end{equation*}
According to \eqref{k_1} this inequality becomes
$$
 \lambda_1(C +1-A(t-t_0)) + A(p-1)\leq 2(p-1)\beta K_1^p C_1^{-1}\frac{ B}{\xi_1^{p-1}}.
$$
One can see that a sufficient condition on $A$, $B$ and $C$ would be
\begin{equation*}\label{cond1ABC}
\lambda_1(C +1) + A(p-1)\leq 2(p-1)\beta K_1^p C_1^{-1}\frac{ B}{\xi_1^{p-1}}.
\end{equation*}
\end{proof}

We will obtain an upper bound of the REF $\phi$ at a certain time $T_1$ up to the boundary as a consequence of comparison of $\phi$ with the barrier function of Lemma \ref{Lemma_barrier}.

\medskip

\begin{lemma}\label{lemma_comp_phi_Phi}
Let $\Phi$ be the barrier function introduced in  Lemma \ref{Lemma_barrier}, given by
\[
\Phi(t,x)=C-BV(x)-At,
\]
Let $\tau_n\rightarrow \infty$ be a sequence along which the REF converges to $1$ as stated above. Then for every $\epsilon>0$ we can choose $n_\epsilon  >0$ and positive constants $A$, $B$, $C$ and $\delta$ as in Lemma \ref{Lemma_barrier} such that
\begin{equation}\label{comp_phi_Phi}
\phi(t+\tau_n,x) \leq \Phi(t,x) ,  \quad \forall x\in \Omega_{\delta} , \ \forall n \geq n_{\epsilon}, \ \forall t\in [0,T_1],
\end{equation}
where
\begin{equation}\label{Tquasilineal}
T_1 =  T_1(\epsilon,\delta)=\frac{C-BC_1^m\delta-\epsilon}{A}.
\end{equation}
\end{lemma}

\begin{proof}

 We fixe $\epsilon>0$ and consider $0<\delta<\xi_1$ where $\xi_1>0$ is given as in Lemma \ref{Lemma_barrier}.
 Also, let $T>0$ and $(\tau_n)$ that we fixed in Assumptions (A).
    By the uniform inner convergence stated in Lemma \ref{lemma_inner_conv} we know there exists $n_{\epsilon ,\delta}>0$ such that
\begin{equation}\label{choice_t0}
|\phi(t+\tau_n,x)| <\epsilon \quad \text{for } x\in \Omega_{I,\delta} , \ t\in[0,T],\  n\geq n_{\epsilon ,\delta}.
\end{equation}
Once we will choose $\delta>0$ we will obtain $n_\epsilon$ as above.

A first condition on the parameters will be that
\begin{equation}\label{cond2ABC}
T_1(\epsilon,\delta)=\frac{C-BC_1^m\delta-\epsilon}{A} < T.
\end{equation}

Now, we consider the barrier function $\Phi$ and prove that $\phi(t+\tau_n,x) \leq \Phi(t,x)$, for a fixed $n \geq n_\epsilon$, on the set $(0,T_1) \times \Omega_{\delta}$, where $T_1=T_1(\epsilon,\delta)$. More exactly, inequality \eqref{comp_phi_Phi} follows as a consequence of the parabolic maximum principle on this set.

 Therefore, we have to check that this comparison is satisfied on the parabolic boundary formed by three pieces
${0}\times \Omega_{\delta} \cup (0,T_1) \times \partial \Omega_{I,\delta} \cup (0,T_1) \times \partial \Omega.$

\noindent $1.$ \emph{Comparison of $\phi$ with $\Phi$ at the initial section $t=0$.} We want to obtain that
\begin{equation}\label{comp_phi_Phi_initial_section}
\phi(\tau_n,x) \leq \Phi(0,x) =C-BV(x)
\end{equation}
for all $x \in \Omega_{\delta}.$
This is possible because of the uniform boundedness of $\phi$
\[
\left(\frac{C_0}{C_1}\right)^m-1=C_{2,m}\leq \phi \leq C_{3,m}=\left(\frac{C_1}{C_0}\right)^m-1,
\]
for all $x \in \Omega$ as a consequence of bounds \eqref{bound_v}. Now, we simply choose $C$ sufficiently large and $A$, $B$ to satisfy \eqref{condABCsupersolution}.

\noindent$2.$ \emph{Comparison of $\phi$ with $\Phi$ on the inner parabolic boundary.} This part of the boundary is given by the set
$$(0,T_1)\times \partial \Omega_{I,\delta}= \{(t,x): t\in (0,T_1), \ x\in \Omega, \ d(x)=\delta \}.$$
For $(t,x)$ as before, $\Phi(t,x)$ is bounded as follows
$$ C-At-BC_1^m \delta \leq\Phi(t,x)\leq C-At-BC_0^m \delta.$$
Let us fixe $\epsilon>0$ and $0<\delta<\xi_1$ where $\xi_1>0$ is given as in Lemma \eqref{Lemma_barrier}. By \eqref{choice_t0}
$$
\phi(t,x) <\epsilon \quad \text{for } x\in \Omega_{I,\delta} , \ t \in [0,T_1].
$$
Thus one can obtain $\phi\leq \Phi$ if
\begin{equation}\label{inner_compar_phi_Phi}
\epsilon \leq C-At-BC_1^m \delta, \quad \forall t \in [0,T_1].
\end{equation}
Since $C$ can not be small, this implies a choice for $A$ and $B$ compatible with \eqref{condABCsupersolution} from the construction of the barrier. This can be realized by choosing $\delta>0$ small. Once $C$ and $B$ are chosen it is sufficient to take $At$ small.

\noindent $3.$ \emph{Comparison of $\phi$ with $\Phi$ on the outer lateral boundary.} This part of the boundary is given by the set
$[0,T_1]\times \partial \Omega$, where we only know that $\phi =v^m/S^m -1$ is bounded. As in \cite{BDGV09} we can use an approximation trick using the solutions $u_k$ of problems posed in the domain $\Omega^k \subset \Omega$. We will prove the desired comparison \eqref{comp_phi_Phi} for the function $u_k$ and obtain it for $u$ by passing to the limit.

We know that $u_k \nearrow u$ as $k \rightarrow \infty$ uniformly on the compact set $[0,t]\times \overline{\Omega}$, for every $t\le T_1.$ Then
$$
\phi_k=\frac{u_k^m}{\mathcal{U}^m}-1=-1<0 \quad \text{on } [0,T_1]\times \partial \Omega^k,
$$
where $\mathcal{U}=e^{-\lambda_1 t}S(x)$ is a separate variables solutions of the DNLE in $\Omega.$ Thus by \eqref{inner_compar_phi_Phi} we have
$$\phi_k<0<C-At=\Phi\quad \text{on } [0,T_1]\times \partial \Omega^k.$$
Steps $1$ and $2$ hold also for $u_k$ since $u_k \leq u$. Thus, by the parabolic comparison principle we obtain that
$\phi_k \leq \Phi$ on the region $\Omega^k \cap \Omega_{\delta}$ for $t\in[0,T_1]$. Passing to the limit when $k\rightarrow \infty$ we obtain $\phi \leq \Phi$ on $[0,T_1]\times \Omega_{\delta}$.

We obtain in this way an improvement of the upper bound of $\phi$  near the boundary after some time delay given by
$$
t \leq T_1=T_1(\epsilon,\delta)=\frac{C-BC_1^m\delta-\epsilon}{A},
$$
which is the maximum that \eqref{inner_compar_phi_Phi} allows. Notice that the delay time $T_1(\epsilon,\delta)$ does not depend on the time $\tau_n$ we fixed at the beginning.

Therefore, in order to choose the desired parameters we perform the following steps: we choose $C$ sufficiently big to have \eqref{comp_phi_Phi_initial_section}. Then choose $A$ and $B$ to satisfy \eqref{condABCsupersolution}.
Finally we choose $\delta$ small such that \eqref{inner_compar_phi_Phi} and \eqref{cond2ABC} hold , that is $t\leq T_1(\epsilon,\delta)\leq T.$
\end{proof}

\medskip

\noindent \textbf{Better estimate from above for $\phi$ up to the boundary}

Under the assumptions of  Lemma \ref{lemma_inner_conv} and Lemma \ref{lemma_comp_phi_Phi} we deduce that for $t= \tau_n + T_1(\epsilon,\delta)$, where $T_1(\epsilon,\delta)$ is given by \eqref{Tquasilineal} and $ n\geq n_\epsilon$, the REF $\phi$ satisfies the upper bound
\begin{equation}
\phi(t,x) \leq \left\{
                 \begin{array}{ll}
                   \epsilon, & \hbox{$d(x)>\delta$;} \\
                   \epsilon+BC_1^m\delta, & \hbox{$d(x)<\delta$.}
                 \end{array}
               \right.
\end{equation}

Therefore, by fixing $\epsilon>0$, finding a barrier with constants $A$, $B$ and $C$ and then taking $\delta <\epsilon/(BC_1^m)$, we obtain the
time $T_1(\epsilon,\delta)$ and the level $n_{\epsilon}$ such that for all $n\ge n_{\epsilon}$ we have
\begin{equation}\label{estim_phi_T1}
\phi(\tau_n+T_1,x) \leq 2\epsilon \quad \forall x \in \Omega.
\end{equation}

This means that $v(T_1+\tau_n)\le (1+\epsilon )S$. The maximum principle implies now that the comparison is valid for all times $t\ge T_1+\tau_{n_\epsilon}$.
This proves that $\overline{c}_{\infty}\le c_*$, thus they are the same. One of the consequences is that $c_*$ does not depend on the subsequence,
therefore the whole family $v(t,\cdot)$ converges to $S=c_* f$ as $t\to\infty$. Moreover, we conclude the uniqueness of the profile $c_*f$ as well as the upper approximation stated in Proposition \ref{pro_behavior_bdry_REF}.

\subsection{Construction of lower barriers }

It remains to prove a similar bound for the REF $\phi$ from below.
To this aim we define
$$
\psi:=-\phi=1-\frac{v^m(t,x)}{S^m(x)}.
$$
We perform a similar approach as in the upper barrier case.

\noindent $\bullet$ \emph{The parabolic equation of $\psi$}
\begin{equation}\label{REFpsi_eq}
-(p-1)(1-\psi)^{p-2}\psi_t= V^{-(p-1)}\Delta_p ((1-\psi)V)+\lambda_1(1-\psi)^{p-1}.
\end{equation}

\noindent  $\bullet$ The function $\psi$ is uniformly bounded in $(t,x)$ for $t\geq 0$. This can be deduced from the estimates \eqref{bound_v} on $v$ and $S$, which is a stationary solution:
\[
1-\left(\frac{C_1}{C_0}\right)^m=C_{2,m}\leq \psi \leq C_{3,m}=1-\left(\frac{C_0}{C_1}\right)^m.
\]

\noindent  $\bullet$ In any interior region $\Omega_{I,\delta} \subset \Omega$, the function $\psi$ satisfies
$$\displaystyle{1-\psi= \frac{v^m}{V}>0} \quad \text{ in }  \Omega_{I, \delta}  \quad \text{for any } t\geq 0.$$

\begin{lemma}\label{Lemma_barrier_Psi}(\textbf{Lower barrier})
We can choose positive constants $A',B',C'$ so that for every $t_0>0$ the function
\begin{equation}\label{barrier_PSI}
\Psi(t,x)=C'+B'V(x)-A'(t-t_0),
\end{equation}
is a super-solution to equation \eqref{REFpsi_eq} on a parabolic region near the boundary
$$\Sigma_{\Psi}=\Sigma_{\Psi,\frac{1}{2}}=\left\{(t,x)\in (t_0,\infty): 0\leq \Psi(t,x)\leq \frac{1}{2} \right\},$$
and moreover $\Sigma_{\Psi}\subset (t_0,T_0) \times \Omega_{\xi_2},$ where $\xi_2 \leq \xi_0$.
\end{lemma}
\begin{proof}
We will prove that the function $\Psi$ given by \eqref{barrier_PSI} is a supersolution for equation \eqref{REFeq} on the parabolic region $\Sigma_{\Psi}$ if we can find constants $A'$, $B'$ and $C'$ such that
\begin{equation}\label{REFsupersolution_Psi}
-(p-1)(1-\Psi)^{p-2}\Psi_t \ge V^{-(p-1)}\Delta_p ((1-\Psi)V)+\lambda_1(1-\Psi)^{p-1}.
\end{equation}
From the beginning we assume that $(t,x) \in \Sigma_{\Psi} \subset (t_0,T_0) \times \Omega_{\xi_2}$ where $T_0$ is such that
\begin{equation}\label{time_interval_barrier_PSI}
\frac{C'+B'V(x)-1/2 }{A'}<T_0-t_0 \leq \frac{C'+B'V(x) }{A'}.
\end{equation}

The left hand side term of \eqref{REFsupersolution_Psi} is positive on $\Sigma_{\Psi}$
\begin{equation}
-(p-1)(1-\Psi)^{p-2}\Psi_t = (p-1)A'(1-\Psi)^{p-2}.
\end{equation}
The right hand side term \eqref{REFsupersolution_Psi} is of the form
$$V^{-(p-1)}\Delta_p ((1-\Psi)V)+ \lambda_1(1-\Psi)^{p-1}=V^{-(p-1)}\Delta_p g(V)+ \lambda_1(1-\Psi)^{p-1}$$
where
$$
g(z)=(1-C-Bz+A(t-t_0))z,
$$
and
  $$
 \Delta_p g(V)=|g'(V)|^{p-2}g'(V) \Delta_p V + (p-1)|g'(V)|^{p-2}g''(V)|\nabla V|^p.
  $$
\newpage

\noindent \textbf{Properties of the function $g$}
\begin{enumerate}
  \item Function $g$ is a concave parabola with zero values at the points $z=0$ and $z=z_0$ where
$$z_0:=\frac{1-C'+A'(t-t_0)}{B'}.$$

  \item The derivatives are
  $$
g'(z)=1-C'-2B'z+A'(t-t_0), \quad g''(z)=-2B'.
$$
\item When applied to $V$, $g'(V(x))$ is positive and bounded sufficiently close the boundary.

More exactly, we consider
\begin{equation}\label{choice2_xi2_PSI}
\xi_2 = \min\left\{\xi_0, \ \frac{1}{C_1^m}\frac{z_0}{4}= \frac{1-C'+A'(t-t_0)}{4B'C_1^m} \right\}.
\end{equation}
For this choice we obtain the following bound on  $\Omega_{\xi_2}$:
$$
0<V(x) \leq C_1^m d(x) \leq C_1^m \xi_2 \leq \frac{z_0}{4}
$$
and then
$$0<\frac{k_1}{2}\leq g'(V(x))\leq k_1,$$
where
\begin{equation}\label{k_1_PSI}
k_1:=g'(0)=1-C+A'(t-t_0), \qquad  g'\left(\frac{z_0}{4}\right)= \frac{1-C'+A'(t-t_0)}{2}=\frac{k_1}{2}.
\end{equation}
\end{enumerate}

\noindent \textbf{Sufficient conditions for the parameters}

Since $V$ is a solution of the stationary problem \eqref{eqV} then $-\Delta_p V=\lambda_1 V^{p-1}$ and therefore the supersolution inequality \eqref{REFsupersolution_Psi} can be rewritten as
\begin{equation}\label{REFsupersolution_Psi_2}
(p-1)A'(1-\Psi)^{p-2}\geq  -\lambda_1 |g'(V)|^{p-2}g'(V) -2B'(p-1)V^{-(p-1)}|g'(V)|^{p-2}|\nabla V|^p +\lambda_1(1-\Psi)^{p-1}.
\end{equation}
Next, the idea is that when we are sufficiently close to the boundary, $V^{-(p-1)}$ will be large enough and then the right hand side term of \eqref{REFsupersolution_Psi_2} will be negative. More exactly, on $\Sigma_{\Psi}$
$$-2B'(p-1)V^{-(p-1)}|g'(V)|^{p-2}|\nabla V|^p +\lambda_1(1-\Psi)^{p-1} \leq -2B'(p-1)(C_1^m \xi_2)^{-(p-1)}k_2 + \lambda_1,$$
where $k_2$ is a positive constant given by
$$|g'(V)|^{p-2}|\nabla V|^p \leq \max\{k_1^{p-2},(k_1/2)^{p-2} \}\cdot  K_2^p =:k_2.$$
Moreover, if $\xi_2$ is sufficiently small
\begin{equation*}
\xi_2 \leq \left( \frac{2B'(p-1)k_2}{C_1 \lambda_1}   \right)^{1/(p-1)}
\end{equation*}
then
$$-2B'(p-1)V^{-(p-1)}|g'(V)|^{p-2}|\nabla V|^p +\lambda_1(1-\Psi)^{p-1} \leq 0$$
and inequality \eqref{REFsupersolution_Psi_2} holds true. We remark that the above conditions on the distance to the boundary $\xi_2$ can be rewritten as
 \begin{equation}\label{choice3_xi2_PSI}
\xi_2 = \min\left\{\xi_0, \ \frac{1-C'+A'(t-t_0)}{4B'C_1^m} , \ \left( \frac{2B'(p-1)k_2}{C_1 \lambda_1}   \right)^{1/(p-1)} \right\}.
\end{equation}
\end{proof}

\begin{lemma}\label{lemma_comp_psi_Psi}
Let $\Psi$ be the barrier function introduced in Lemma \ref{Lemma_barrier_Psi}, given by
\[
\Psi(t,x)=C'+B'V(x)-A't.
\]
Let $\tau_n \rightarrow \infty$ be a sequence along which the REF converges to $1$ as stated in Assumptions (A). Then for any $\epsilon>0$ we can choose $n_\epsilon>0$ and positive constants $A'$, $B'$, $C'$ and $\delta$ as in Lemma \ref{Lemma_barrier_Psi} such that
\begin{equation}\label{comp_psi_Psi}
\psi(t+\tau_n,x) \leq \Psi(t,x) ,  \quad \forall x\in \Omega_{\delta} , \ \forall n \geq n_{\epsilon},  \  \forall t\in [0,T_2],
\end{equation}
where
\begin{equation}\label{Tquasilineal2}
T_2 = T_2(\epsilon,\delta)=\frac{C'+B'C_0^m\delta-\epsilon}{A'}.
\end{equation}
\end{lemma}

\begin{proof}
 We fixe $\epsilon>0$ and consider $0<\delta<\xi_2$ where $\xi_2>0$ is given as in Lemma \ref{Lemma_barrier_Psi}.
 Also, let $T>0$ and $(\tau_n)$ that we fixed in Assumptions (A).

We adapt the proof of Lemma \ref{Lemma_barrier} by writing the estimates in terms of the function $\psi=-\phi$.  By the uniform inner convergence stated in Lemma \ref{lemma_inner_conv} we know there exists $n_\epsilon>0$ such that
\begin{equation}\label{ineq_psi_eps}
|\psi(\tau_n+t,x)| <\epsilon \quad \text{for } x\in \Omega_{I,\delta} , \ t\in[0,T],\  n\geq n_\epsilon.
\end{equation}

We impose a first condition on the parameters $A'$, $B'$, $C'$
\begin{equation}\label{T2epsdelta}
T_2(\epsilon,\delta)=\frac{C'+B'C_0^m\delta-\epsilon}{A'}<T.
\end{equation}

Now, we consider the barrier function $\Psi$ constructed in Lemma \ref{Lemma_barrier_Psi} and prove that $\psi(t+\tau_n,x) \le \Psi(t,x)$, for a fixed $n \geq n_\epsilon$, on the set $(0,T_2) \times \Omega_{\delta}$, where $T_2=T_{2,\epsilon,\delta}$. More exactly, inequality \eqref{comp_psi_Psi} follows as a consequence of the parabolic maximum principle on this set.

Therefore, we have to check that this comparison is satisfied on the parabolic boundary formed by three pieces
${0}\times \Omega_{\delta} \cup (0,T_2) \times \partial \Omega_{I,\delta} \cup (0,T_2) \times \partial \Omega.$

\noindent $1.$ \emph{Comparison of $\psi$ with $\Psi$ at the initial section $t=0$.} We want to obtain that
$$\psi(\tau_n,x) \leq \Psi(0,x) =C'+B'V(x)$$
for all $x \in \Omega_{\delta}.$
This is possible because of the uniform boundedness of $\psi$
$$
1-\left(\frac{C_1}{C_0}\right)^m=C_{3,m}\leq \psi \leq C_{4,m}=1-\left(\frac{C_0}{C_1}\right)^m,
$$
for all $x \in \Omega$ as a consequence of bounds \eqref{bound_v}. Now, we simply choose $C'>1$.

\noindent$2.$ \emph{Comparison of $\psi$ with $\Psi$ on the inner parabolic boundary.} This part of the boundary is given by the set
$$(0,T_2)\times \partial \Omega_{I,\delta}= \{(t,x): t\in (0,T_2), \ x\in \Omega, \ d(x)=\delta \}.$$
For $(t,x)$ as before, $\Psi(t,x)$ is bounded as follows
$$ C'-A't+B'C_0^m \delta \leq \Phi(t,x)\leq C'-A't+B'C_1^m \delta.$$
Let us fixe $\epsilon>0$ and $0<\delta<\xi_2$ where $\xi_2>0$ is given as in Lemma \eqref{Lemma_barrier_Psi}. By \eqref{ineq_psi_eps}
$$
\psi(t+\tau_n,x) <\epsilon \quad \text{for } x\in \Omega_{I,\delta} , \ t \in [0,T_2].
$$
Thus one can obtain $\psi\leq \Psi$ if
\begin{equation}\label{inner_compar_psi_Psi}
\epsilon \leq   C'-A't+B'C_0^m \delta, \quad \forall t \in [0,T_2],
\end{equation}
or equivalently
$$ t \leq \frac{C'+B'C_0^m\delta-\epsilon}{A'}=T_2(\epsilon,\delta).$$

\noindent $3.$ \emph{Comparison of $\psi$ with $\Psi$ on the outer lateral boundary.} This part of the boundary is given by the set
$[0,T_2]\times \partial \Omega$, where we only know that $\psi =1-v^m/S^m$ is bounded. Here we will use an approximation trick using the solutions $u_k$ of problems posed in an extended domain $\Omega^k \supset \Omega$. Like in Lemma \ref{lemma_comp_phi_Phi} we prove the desired comparison \eqref{comp_psi_Psi} for the approximating function $u_k$. We know that $u^k \searrow u$ as $k \rightarrow \infty$ uniformly on the compact set $[0,t]\times \overline{\Omega}$, for every $t\le T_2.$
Formally
$$
\psi_k=1-\frac{u_k^m}{\mathcal{U}^m}=-\infty \quad \text{on }[0,T_2]\times \partial{\Omega},
$$
where $\mathcal{U}=e^{-\lambda_1 t}S(x)$ is a separate variables solutions of the DNLE in $\Omega.$ On the other hand, on $[0,T_2]\times \partial \Omega$,
$$
\Psi(t,x)= C'-A't \geq C'- A' T_2=\epsilon-BC_0^m\delta
$$
and thus the comparison $\psi_k\leq \Psi$ holds true on the outer lateral boundary. Finally, steps $1$ and $2$ hold also for $u_k$ since $u_k\geq u$ and thus, by parabolic comparison we obtain that $\phi_k\leq \Phi$ on $[0,T_2]\times \Omega_{\delta}.$ Passing to the limit when $k \rightarrow \infty$ we obtain $\psi \leq \Psi$ on $[0,T_2]\times \Omega_{\delta}.$

We obtain in this way the an improvement of the lower bound of $\phi$ near the boundary after some time delay given by
$$
t \leq T_2(\epsilon,\delta)=\frac{C'+B'C_0^m\delta-\epsilon}{A'},
$$
which is the maximum that \eqref{inner_compar_psi_Psi} allows. Notice that the delay time $T_2(\epsilon,\delta)$ does not depend on the time $\tau_n$ we fixed at the beginning.

Therefore, in this case of comparison with lower barriers, in order to choose the desired parameters we perform the steps: we choose $C'>1$ big enough to have the inner comparison and then we choose $A'$, $B'$ such that condition \eqref{inner_compar_psi_Psi} holds. Finally choose $\delta$ small enough such that condition \eqref{T2epsdelta} holds.

\end{proof}

\medskip

\noindent \textbf{Better estimate from below for $\phi$ up to the boundary}

Under the assumptions of  Lemma \ref{lemma_inner_conv} and Lemma \ref{lemma_comp_psi_Psi} we deduce that for $t= \tau_n + T_2(\epsilon,\delta)$, for a fixed $n \geq n_\epsilon$, where $T_2(\epsilon,\delta)$ is given by \eqref{Tquasilineal2} the REF $\psi=-\phi$ satisfies the upper bound
\begin{equation}
\psi(t,x) \leq \left\{
                 \begin{array}{ll}
                   \epsilon, & \hbox{$d(x)>\delta$;} \\
                   \epsilon-B'C_0^m\delta, & \hbox{$d(x)<\delta$.}
                 \end{array}
               \right.
\end{equation}

Therefore, by fixing $\epsilon>0$, finding a barrier $\Psi$ with constants $A'$, $B'$ and $C'$ and then taking $\delta <\epsilon/(B'C_0^m)$, we obtain a time $T_2(\epsilon,\delta)$  and a $n_\epsilon$ such that for all $n\ge n_{\epsilon}$ we have
\begin{equation}\label{estim_psi_T2}
\psi(\tau_n+T_2,x) \leq \epsilon \quad \forall x \in \Omega.
\end{equation}

An immediate consequence is the lower estimate of $v$ in terms of the profile $v(T_2+\tau_n)\ge (1-\epsilon )S$. The maximum principle implies now that the comparison is valid for all times $t\ge T_2+\tau_{n_\epsilon}$. This proves that $\underline{c}_{\infty}\ge c^*$, so that they are the same.
Therefore we proved that
$$
\underline{c}_{\infty}= c^*=\overline{c}_{\infty}.
$$
This concludes the part of the lower approximation stated in Proposition \ref{pro_behavior_bdry_REF}.

Once we proved the uniqueness of the asymptotic profile and, Theorem \ref{Th_conv_Rel_error_quasilinear} follows as a consequence of the estimates
\eqref{estim_phi_T1} and \eqref{estim_psi_T2}.

\begin{figure}[!h]\label{figure_Barriers}
	    \centering
\subfigure[Behaviour of $\Phi(t,x)$]{\includegraphics[width=65mm,height=50mm]{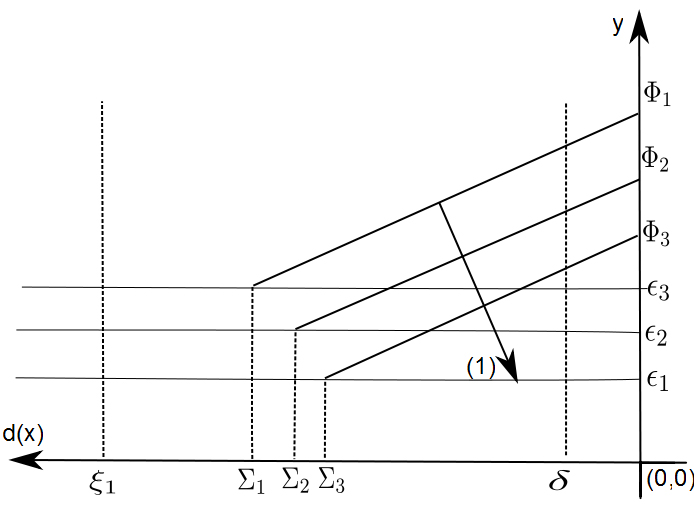}} 
\qquad
\subfigure[Behaviour of $\Psi(t,x)$]{\includegraphics[width=65mm,height=50mm]{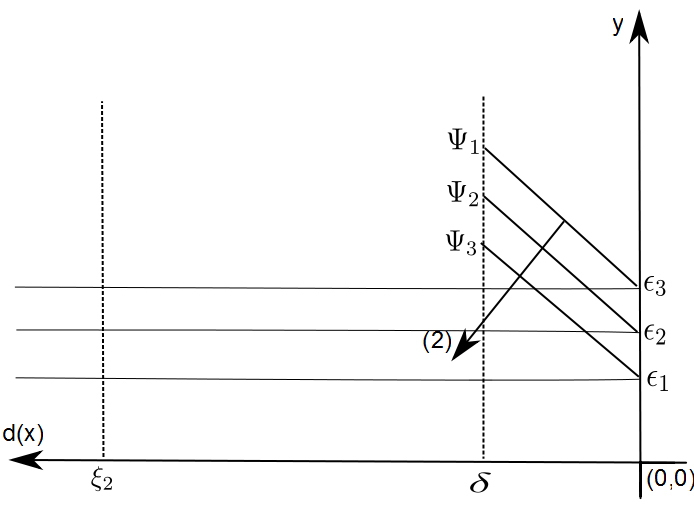}}
\caption{\small{Idea of the behaviour of the barriers: $y$-axis: values of $\Phi(t,x)$, $x$-axis: values of $d(x)=d(x,\partial \Omega)$,i.e. the distance
to the boundary. $\Sigma_i$: the points where the barrier (a) $\Phi(t,x)=\epsilon_i$, (b)$\Psi(t,x)=\epsilon_i$, $i=1,2,3$. $\epsilon_i$:
different values of $\epsilon$( decreasing with i=1,2,3) give different barriers $\Phi_i$, $\Psi_i$ decreasing with $\epsilon$ as the arrows (1) and (2) indicates.  $\xi_1$ and $\delta$ as in Lemma \ref{Lemma_barrier}, $\xi_2$ as in Lemma \ref{Lemma_barrier_Psi}.} }
\end{figure}
\end{section}
\begin{section}{Appendix}

\subsection{Two convergence results}
The following lemma can be easily proved with basic computations.

\medskip

\begin{lemma}\label{Lemma_conv_measure}{\textbf{(Property of the convergence in measure)}}
Let $(f_n)_n, f \subset L^{p}(\Omega)$ a sequence of functions such that
\begin{itemize}
\item $f_n \rightarrow f$ in measure;
\item $\|f_n\|_{L^p(\Omega)}$ uniformly bounded.
\end{itemize}
Then $$f_n \rightarrow f \text{ in }L^q(\Omega), \quad \text{ for every  } 1\leq q <p.$$
\end{lemma}
Another useful result in our proofs is a lemma concerning nonlinear monotone operators due to Brezis \cite{BrezisMonotonicity}.
\begin{lemma}\label{brezis}
Let $A$ be a maximal monotone operator on a Hilbert space $H$.
Let $Z_n$ and $W_n$ be measurable functions from $\Omega$ (a finite measure space) into $H$.
Assume $Z_n \rightarrow Z$ a.e. on $\Omega$ and $W_n \rightharpoonup W$ weakly in $L^1(\Omega;H)$.
If $W_n(x) \in A(Z_n(x))$ a.e. on $\Omega$, then $W(x) \in A(Z(x))$ a.e. on $\Omega$.
\end{lemma}

\subsection{Regularity}

\noindent Concerning the regularity of the solution $u$ of the DNLE, we refer for example to \cite{Ivanov}, \cite{MR1218742}, \cite{MR1156216}.

\noindent \textbf{(Theorem 2.1 from \cite{Ivanov}- inner H\"{o}lder estimate)}  Let $u$ be a weak solution of the DNLE. Then
   \begin{equation}\label{regularity_DNLE}
  u \in C_{loc}^{\alpha/p, \alpha}([0,T]\times \Omega) \text{ for some }\alpha \in (0,1).\end{equation}
   Moreover, for every cylinder $Q'=[\epsilon,T]\times \Omega'$, $\overline{\Omega'} \subset \Omega, \ \epsilon>0$, we have
      \begin{equation}\label{estim_cylinder}
      \sup_{(t,x)(t',x')\in  Q'} \frac{|u(t,x)-u(t',x')|}{(|t-t'|^m+|x-x'|^m )^{\alpha/p}}\leq K,
      \end{equation}
      where $\alpha \in (0,1)$ and $K>0$ depend only on the $T$, $\Omega'$ and the data.

\noindent \textbf{(Theorem 2.2. from \cite{Ivanov}- H\"{o}lder estimate up to the boundary)}  If $\Omega$ has regular boundary then
  $$u \in C_{loc}^{\alpha/p, \alpha}([0,T]\times \Omega) \text{ for some } \alpha \in (0,1)$$
   and $u$ satisfies an estimate similar to \eqref{estim_cylinder}.

\subsection{Distance to the boundary function}\label{subsection_dist_bdry}

We collect some properties of the \emph{distance to the boundary function} for which we refer to \cite{Trudinger} and \cite{MR0282058}.
Let $d:\overline{\Omega}\rightarrow [0,+\infty)$ be given by
$$d(x)=\text{dist}(x,\partial \Omega)=\min\{|x-z|:z \in \partial \Omega\},$$
where $|\cdot|$ is the Euclidean norm of $\mathbb{R}^d$.
In terms of $d(x)$ we define the sets
$$\Omega_{I,r} = \{x \in \overline{\Omega}:d(x)> r\},$$
$$\Omega_{r}=\Omega \setminus \overline{\Omega_{I,r}}= \{x \in \overline{\Omega}:d(x)< r\},$$
and we remark that, for all small $r>0$,
$$\partial \Omega_{I,r}=\{x \in \Omega: d(x)= r \}.$$

\begin{lemma}\label{properties_d(x)}(\textbf{Properties of the distance to the boundary})
Let $\Omega\subset \mathbb{R}^N$ be a bounded domain with boundary $\partial \Omega$ of class $C^2$. Then
\begin{enumerate}
  \item there is a constant $\xi_0 \in \mathbb{R}_+$ such that for every $x \in \Omega_{\xi_0}$, there is a unique $h(x) \in \partial \Omega$ which realizes the distance
$$d(x)=|x-h(x)|.$$
Moreover, $d(x) \in C^2(\Omega_{\xi_0}),$ and for all $r\in [0,\xi_0)$ the function $H_r:\partial ( \overline{\Omega_r}) \cap \Omega \rightarrow \partial \Omega$ defined by $H_r(x)=h(x)$ is a homeomorphism.
  \item Function $d(x)$ is Lipschitz with constant $1$, i.e.
  $$|d(x)-d(y)| \leq |x-y|.$$
  Moreover,
  \begin{equation}\label{estim_grad_d(x)}
  0<c \leq |\nabla d(x)|\leq 1, \quad \forall x\in \Omega_{\xi_0},
  \end{equation}
  and there exists a constant $K>0$ such that
$$-K \leq \partial^2_{ij} d(x) \leq K ,   \quad \forall x \in {\Omega}_{\xi_0}, \forall i,j=1,N.$$
\end{enumerate}
\end{lemma}
\noindent Notice that this $\xi_0$ can be characterized as follows:
$$\displaystyle{\xi_0=\{ \min_{\overline{x} \in \partial \Omega} \max_{r>0} r : B_r(\overline{x}+r  \nu) \text{ is tangent at }\partial \Omega \text{ in }\overline{x}  \},}$$
where $\nu$ is the inward unit normal at $\partial \Omega$ in $x_0$. We observe that $d(x_0+r \nu)=R$ and
\begin{equation}\label{boundary_property}
\Omega_r \subset \bigcup_{y \in \partial \Omega_{I,r}} B_r(y).
\end{equation}

Throughout the paper we have constantly used the notation $\xi_0$ with the properties stated above.

\end{section}

\section{Comments and open problems}

\noindent $\bullet$ In this paper we have discussed only the case slow diffusion case $m(p-1)>1$ and the quasilinear case $m(p-1)=1$. The fast diffusion case $m(p-1)<1$ produces different results and thus it needs different techniques. For this last case we mention the results of Savar\'{e} and Vespri (\cite{MR1285092}) about the asymptotic behaviour of the DNLE in the singular case. In that paper the authors prove the convergence to an asymptotic profile for a sequence of times $t_n \rightarrow T$ , $T$ being the extinction time. The uniform convergence for all times and the rate of convergence in the fast diffusion case, for both DNLE and PLE, remain an open problem at this moment. However, the fast diffusion regimes for the PME and the PLE have been much discussed in the literature, we mention \cite{BDGV09, FeireislSimondon}.

\noindent $\bullet$ We presented only a formal description of the self similar solutions of the DNLE in the case $m(p-1)>1$. We do not offer a complete characterization of such functions since this beyond the purpose of our paper. The problem is interesting and it deserves a separate study itself.

\noindent $\bullet$ Our result in the quasilinear case is not as sharp as the result in the degenerate case. Indeed, we only prove convergence in relative
error. The problem of a rate of convergence similar to Theorem \ref{Th_rate_conv_DNLE} is still open, except in the linear case  $m=1$, $p=2$, where a
representation as infinite series follows from the Fourier analysis of the solution.

\noindent $\bullet$ For the Cauchy problem we mention \cite{MR2602927} where the authors prove an $L^1$-algebraic decay of the non-negative solution to a Barenblatt-type solution for the case $\frac{N-p}{N(p-1)}<m<\frac{N-p+1}{N(p-1)}$, and they estimate its rate of convergence.

\noindent $\bullet$ More general problems of this type can be considered by similar techniques. Let us mention the doubly nonlinear equation with mixed boundary conditions or $p$-Laplacian type equations with variable coefficients.

\noindent {\large\bf Acknowledgments}

\noindent Both authors partially supported by the
Spanish Project MTM2008-06326-C0-01.

\noindent \textbf{Keywords.} doubly nonlinear equation, slow diffusion, asymptotic behaviour, self-similar solution, convergence rates.

\noindent \textbf{Mathematics Subject Classification.} 35B40, 35B45, 35B65, 35R35, 35K55, 35K65.

\hfill

\noindent (a) Diana Stan. Consejo Superior de Investigaciones Cientificas, Instituto de Ciencias Matematicas, C/Nicolas Cabrera, 13-15, Campus de Cantoblanco, 28049, Madrid, Spain. E-mail address:{{\tt~diana.stan@icmat.es}.\\
AND\\
Departamento de Matem\'{a}ticas, Universidad
Aut\'{o}noma de Madrid, Campus de Cantoblanco, 28049 Madrid, Spain. \\

\noindent (b) Juan Luis V\'azquez. Departamento de Matem\'{a}ticas, Universidad
Aut\'{o}noma de Madrid, Campus de Cantoblanco, 28049 Madrid, Spain.
E-mail address:{{\tt~juanluis.vazquez@uam.es}.

\end{document}